\documentclass[12pt,oneside,english]{amsart}
\usepackage{lmodern}
\usepackage[utf8]{inputenc}
\usepackage[english]{babel}
\usepackage{geometry}
\usepackage{enumitem}
\usepackage{mathrsfs}
\usepackage{amstext}
\usepackage{amsthm}
\usepackage{amssymb}

\usepackage[svgnames]{xcolor}
\usepackage[bookmarksnumbered=true]{hyperref} 
\hypersetup{
	colorlinks = true,
	linkcolor = Blue,
	anchorcolor = blue,
	citecolor = Green,
	filecolor = blue,
	urlcolor = FireBrick
}
\usepackage{bbm}
\usepackage{comment}
\usepackage{graphicx}    
\usepackage{subcaption}  
\usepackage{algorithm}
\usepackage{algorithmic,eqparbox,array}

\usepackage{orcidlink}

\usepackage{aliascnt}
\usepackage[nameinlink]{cleveref}

\makeatletter
\numberwithin{equation}{section}
\numberwithin{figure}{section}

\theoremstyle{plain}
\newtheorem{theorem}{Theorem}[section]

\theoremstyle{definition}
\newaliascnt{definition}{theorem}

\aliascntresetthe{definition}
\crefname{definition}{Definition}{Definitions}

\theoremstyle{definition}
\newaliascnt{question}{theorem}

\aliascntresetthe{question}
\crefname{question}{Question}{Questions}

\theoremstyle{definition}
\newaliascnt{remark}{theorem}
\newtheorem{remark}[remark]{Remark}
\aliascntresetthe{remark}
\crefname{remark}{Remark}{Remarks}

\theoremstyle{plain}
\newaliascnt{corollary}{theorem}

\aliascntresetthe{corollary}
\crefname{corollary}{Corollary}{Corollaries}

\theoremstyle{plain}
\newaliascnt{lemma}{theorem}
\newtheorem{lemma}[lemma]{Lemma}
\aliascntresetthe{lemma}
\crefname{lemma}{Lemma}{Lemmas}

\theoremstyle{plain}
\newaliascnt{proposition}{theorem}
\newtheorem{proposition}[proposition]{Proposition}
\aliascntresetthe{proposition}
\crefname{proposition}{Proposition}{Propositions}

\theoremstyle{definition}
\newaliascnt{example}{theorem}

\aliascntresetthe{example}
\crefname{example}{Example}{Examples}

\theoremstyle{definition}
\newaliascnt{assumption}{theorem}
\newtheorem{assumption}[assumption]{Assumption}
\aliascntresetthe{assumption}
\crefname{assumption}{Assumption}{Assumptions}

\theoremstyle{definition}
\newaliascnt{problem}{theorem}

\aliascntresetthe{problem}
\crefname{problem}{Problem}{Problems}

\newtheorem*{definition*}{Definition}
\newtheorem*{proposition*}{Proposition} 
\newtheorem*{remark*}{Remark}
\newtheorem*{example*}{Example}
\newtheorem*{problem*}{Problem}
\newtheorem*{question*}{Question}
\newtheorem*{theorem*}{Theorem}
\newtheorem*{lemma*}{Lemma}
\newtheorem*{corollary*}{Corollary}
\newtheorem*{Acknowledgments*}{Acknowledgments}

\newlist{casenv}{enumerate}{4}
\setlist[casenv]{leftmargin=*,align=left,widest={iiii}}
\setlist[casenv,1]{label={{\itshape\ \casename} \arabic*.},ref=\arabic*}
\setlist[casenv,2]{label={{\itshape\ \casename} \roman*.},ref=\roman*}
\setlist[casenv,3]{label={{\itshape\ \casename\ \alph*.}},ref=\alph*}
\setlist[casenv,4]{label={{\itshape\ \casename} \arabic*.},ref=\arabic*}
\setitemize{leftmargin=*, align=left}
\providecommand{\casename}{Case}

\renewenvironment{proof}[1][\proofname]{\medskip \noindent {\bfseries #1. }}{\hfill \qedsymbol \medskip} 

\usepackage{babel}

\crefformat{subfigure}{#2\figurename~#1#3}
\Crefformat{subfigure}{#2\figurename~#1#3}

\DeclareRobustCommand{\SkipTocEntry}[5]{}

\newcommand{\mR}{\mathbb{R}}   
\newcommand{\mZ}{\mathbb{Z}}   
\newcommand{\mN}{\mathbb{N}}   
\newcommand{\mP}{\mathbb{P}}   
\newcommand{\eps}{\varepsilon}
\newcommand{\e}{\varepsilon}

\newcommand{\rmd}{\mathrm{d}} 
\newcommand{\mL}{\mathcal{L}} 

\newcommand{\mF}{\mathscr{F}} 
\newcommand{\mO}{\mathcal{O}} 
\newcommand{\mD}{\mathcal{D}} 

\newcommand{\one}{\mathbbm{1}}
\newcommand{\bbE}{\mathbb{E}}

\newcommand{\calG}{\mathcal{G}}

\newcommand{\Om}{\Omega}

\newcommand{\abs}[1]{\lvert #1 \rvert}  
\newcommand{\norm}[1]{\lVert #1 \rVert}  

\newcommand{\KL}{\mathrm{KL}}

\DeclareMathOperator{\supp}{supp}
\DeclareMathOperator{\ssupp}{sing\ supp}
\DeclareMathOperator{\dist}{dist}

\DeclareRobustCommand{\SkipTocEntry}[5]{}

\makeatother

\newcommand{\R}{{\mathbb R}}

\begin{document}
\title[Bayesian inference for the fractional Calder{\'o}n problem]{Bayesian inference for the fractional Calder{\'o}n problem with a single measurement}

\author[P.-Z. Kow]{Pu-Zhao Kow\,\orcidlink{0000-0002-2990-3591}}
\address{Department of Mathematical Sciences, National Chengchi University, Taipei 116, Taiwan}
\email{pzkow@g.nccu.edu.tw}

\author[J. Nurminen]{Janne Nurminen\,\orcidlink{0000-0002-9018-803X}}
\address{Computational Engineering, School of Engineering Sciences, Lappeenranta-Lahti University of Technology, Finland \& Department of Mathematics and Statistics, University of Jyv\"askyl\"a, Finland}
\email{janne.nurminen@lut.fi; janne.s.nurminen@jyu.fi} 

\author[J. Railo]{Jesse Railo\,\orcidlink{0000-0001-9226-4190}}
\address{Computational Engineering, School of Engineering Sciences, Lappeenranta-Lahti University of Technology, Finland} 
\email{jesse.railo@lut.fi} 

\begin{abstract}
\begin{sloppypar}
This paper investigates the consistency of a posterior distribution in the single-measurement fractional Calder{\'o}n problem with additive Gaussian noise. We consider a Bayesian framework with rescaled and Gaussian sieve priors, using a collection of noisy, discrete observations taken from a suitable exterior domain. Our main result shows that the posterior distribution concentrates around the true parameter as the number of measurements increases. Furthermore, we establish tight convergence rates for the reconstruction error of the posterior mean. 
A central technical challenge is to obtain refined stability estimates for both the forward and inverse problems. In particular, the required forward estimates are delicate to obtain because the fractional elliptic problems do not enjoy as strong regularity theory as their classical counterparts.
\end{sloppypar}
\end{abstract}


\maketitle

\begin{sloppypar}

\tableofcontents

\newpage

\section{Introduction}

Let $\Omega$ be a nonempty bounded smooth domain in $\mathbb{R}^d$, and let $(-\Delta)^s$ denote the fractional Laplacian of order $s \in (0,1)$. The fractional Laplacian has many equivalent definitions under suitable interpretations \cite{Kwa17FractionalEquivalent}. 
We fix any $0 \not\equiv \phi \in C_{c}^{\infty}(\Omega_{e})$ with $\Omega_{e}:=\mR^{d}\setminus\overline{\Omega}$, and consider the following Dirichlet problem
\begin{equation}\label{eq_main_intro}
((-\Delta)^{s}+f)u=0 \text{ in $\Omega$} ,\quad u|_{\Omega_{e}}=\phi.
\end{equation}
In this work, we study an inverse problem for the fractional Schr\"{o}dinger equation, which is concerned with the recovery of the potential $f$ from the exterior Dirichlet-to-Neumann map (DN map)
\begin{equation}\label{DN_map}
\Lambda_{f}\phi\equiv\left.(-\Delta)^{s}u_{f}\right|_{\mD}, 
\end{equation}
where $\mD\subset\Om_e$ is a fixed open set. It was proved in \cite{GSU20Calderon} that the potential can be uniquely determined from exterior measurements and, remarkably, a single measurement suffices \cite{GRSU20Reconstruction}.  

We consider a statistical model in which we randomly choose finitely many uniformly distributed points from $\mD$. A priori, the DN map in \eqref{DN_map} is not pointwise well-defined but \Cref{lem:qualitative} shows it to be smooth when restricted to $\mD$ for a smooth exterior value $\phi$. The values of the DN map at these sampled points, with added noise, form our measurement model. We model $f$ with rescaled Gaussian priors and Gaussian sieve priors. We study the concentration and contraction rates, as well as their optimality, of posterior distributions arising from these priors. This approach to studying the Calder{\'o}n problem was initiated in \cite{abraham2020statistical} where the authors introduced a statistical inverse problem for the isotropic conductivity equation. Our aim here is to show that the Bayesian method has theoretical guarantees for the fractional Calder{\'o}n problem as well.

In the work \cite{giordano2020consistency}, the approach from \cite{abraham2020statistical} was generalized to the context of determining a conductivity function from a source-to-solution map. The authors of \cite{giordano2020consistency} used priors that were inspired by previous consistent inversion of noisy non-abelian $X$-ray transforms \cite{MNP21ConsistencyInversion}. Further results for various inverse problems with similar techniques can be found in \cite{FKW24ProbabilityInverseScattering,GK20BVM,Kekkonen2022IP,KW25BayesSubdiffusion,NvdGW20JUQ}, and the references therein.

The nonlocal Dirichlet problem \eqref{eq_main_intro} has been extensively studied in \cite{Gru15FractionalLaplacianDomains, Ros-Oton_survey}. Nonlocal problems of this type also arise naturally in various applications. For instance, nonlocal models describe animal movement \cite{viswanathan1996}, foraging patterns of marine predators \cite{Humphries2010_un}, the dynamics of competing populations sharing the same resource \cite{valdinoci_biology} and in quantum optics \cite{KS2022collectivespontaneous, KS2023kineticequations, OKSW2024nonlocalpartial}.  Moreover, the fractional Laplacian appears in the linear problem of unique continuation from partial data Radon transforms in integral geometry and X-ray tomography \cite{ilmavirta2020unique, CMR20unique, IKS25UCPMRT}. Nonlocal problems also play a central role in physics, fluid dynamics, finance, and image processing, see \cite{Ros-Oton_survey, bucur_valdinoci} for further references. For inverse problems related to \eqref{eq_main_intro}, we mention \cite{CMRU2022higherorder, FGKU2025fractionalanisotropic, feizmohammadi2025fractional} and the monograph \cite{LL25FractionalCalderonBook}. Further nonlinear variants are studied, for example, in \cite{MRZ2023pbiharmonic, KMS23GlobalUniquenessSemilinear, KW23FractionalNonlinear}. 

Our main results show that an estimator, here we use the posterior mean, converges to the ground truth with high probability in the large sample limit. The convergence rate is logarithmic and is based on the stability of the inverse problem \cite{Rueland2021SingleMeasurement}. Furthermore, we prove that this convergence rate is optimal in the statistical minimax sense, which in turn is based on the instability result for the inverse problem \cite{RS18Instability} (see also \cite{Man01instability} for the local version and \cite{KRS21InstabilityMechanism} for a more general setting for instability).
To prove these results, we follow the general framework developed in \cite{giordano2020consistency}. In order to use this framework, we need to analyze the forward and inverse problem in more detail.

For the forward problem, we require the forward operator $f\mapsto\Lambda_f\phi$ to be uniformly bounded on its domain and a suitable $L^2$-based local Lipschitz forward estimate (see \Cref{sec_forward_estimates}). The uniform boundedness is proved via the singular integral definition of the fractional Laplacian. To prove the forward estimate, we analyze the DN map using the following observations. First, we use that the DN map is smooth in the exterior domain when we assign a smooth exterior value, and thus we can use the strong definition of the DN map (see \cite{GSU20Calderon}). We use this together with the weak definition of the DN map and the Alessandrini identity. The key observation is that we can choose up to three different functions to appear in the Alessandrini identity. Choosing these in a clever way, using a stronger elliptic estimate from \cite{covi2022globalinversefractionalconductivity}, and combining this with an approximation argument using cutoffs gives the desired estimate. See \Cref{lem:forward-estimate} for further details.

For the inverse problem, we need a stability estimate for the single measurement Calder\'{o}n problem (for full data this was proven in \cite{RS20Calderon}). A stability estimate is proved in \cite{Rueland2021SingleMeasurement}, but we need to refine it in order to make visible the expressions for $\norm{f_{1}}_{C^{s}(\Omega)}$ and $\norm{f_{2}}_{C^{s}(\Omega)}$. Deriving this estimate requires modifying some of the arguments in \cite{GRSU20Reconstruction,Rueland2021SingleMeasurement} and using tools developed in these articles.
Once we have obtained the required estimates, we prove our main results in the spirit of \cite{giordano2020consistency}.

The rest of the article is organized as follows. In \Cref{sec_main_results}, we introduce more precisely the fractional Schr\"{o}dinger equation and the statistical model for our inverse problem and furthermore state the main results of our work. \Cref{sec_estimates} is devoted to establishing the estimates required for the forward and inverse problems. The proofs of the main results are presented in \Cref{sec_proof_of_main}. In \Cref{sec:MCMC}, we illustrate numerically the convergence of the Bayesian procedure. Finally, we provide the proofs for the results in \Cref{sec_estimates} in \Cref{sec_appendix}.

\subsection*{Notations}
\addtocontents{toc}{\SkipTocEntry}

Throughout this paper, we use the symbols $\lesssim$ and $\gtrsim$ for inequalities holding up to a universal constant. For two real sequences $(a_{N})$ and $(b_{N})$, we write that $a_N\simeq b_N$ if both $a_{N} \lesssim b_{N}$ and $b_{N} \lesssim a_{N}$ hold for all sufficiently large $N$. For a sequence of random variables $Z_N$ and a real sequence $(a_{N})$, we write $Z_N=O_{\Pr}(a_N)$ if for all $\eps>0$ there exists $M_\eps<\infty$ such that for all $N$ large enough, $\Pr(|Z_N|\ge M_\eps a_N)<\eps$. Denote by ${\mathcal L}(Z)$ the law of a random variable $Z$. We also denote $a\vee b=\max\{a,b\}$ for all $a,b\in\mR$. 

Let $d\ge 1$ be an integer. For each $\gamma\in\mR$, let $H^{\gamma}(\mR^{d})$ be the standard $L^{2}$-based fractional Sobolev space. For each open set $U\subset\mR^{d}$, we denote by $H_{0}^{\gamma}(U)$ the closure of $C_{c}^{\infty}(U)$, where $C_{c}^{\infty}(U)$ is the collection of smooth functions supported in $U$, and by $H^{\gamma}(U)$ the restriction of $H^{\gamma}(\mR^{d})$ to $U$, equipped with the standard quotient norm. The following identifications are well-known for any bounded Lipschitz domain $U$: 
\begin{equation*}
(H_{0}^{\gamma}(U))' = H^{-\gamma}(U) ,\quad (H^{\gamma}(U))' = H_{0}^{-\gamma}(U) \quad \text{for all $\gamma\in\mR$,}  
\end{equation*}
see, for example, \cite[Section~2A]{GSU20Calderon}, \cite[Chapter~3]{McL00EllipticSystems} or \cite{KLW22CalderonFractionalWave}, and the references therein for more details.

\section{Main results}\label{sec_main_results}

Let $\Omega$ be a nonempty bounded smooth domain in $\mathbb{R}^d$, and let $(-\Delta)^s$ denote the fractional Laplacian of order $s \in (0,1)$. 
We fix any $0 \not\equiv \phi \in C_{c}^{\infty}(\Omega_{e})$ with $\Omega_{e}:=\mR^{d}\setminus\overline{\Omega}$, and consider the following Dirichlet problem 
\begin{equation}
((-\Delta)^{s}+f)u=0 \text{ in $\Omega$} ,\quad u|_{\Omega_{e}}=\phi, \label{eq:Dirichlet-problem}
\end{equation}
where $\Omega_{e}:=\mR^{d}\setminus\overline{\Omega}$. We assume that $0 \le f\in L^{\infty}(\Omega)$, which ensures the existence and uniqueness of a solution $u=u_{f}\in H^{s}(\mR^{d})$ to \eqref{eq:Dirichlet-problem}, where $H^{s}(\mR^{d})$ is the $L^{2}$-based Hilbert space defined via Fourier transform \cite[Lemma~2.3]{GSU20Calderon}. The nonnegativity of $f$ plays a crucial role in ensuring the $L^{\infty}$-bound 
\begin{equation}
\norm{u_{f}}_{L^{\infty}(\Omega)} \le \norm{\phi}_{L^{\infty}(\Omega_{e})}, \label{eq:L-infty-bound}
\end{equation}
which is a consequence of the maximum principle \cite[Proposition~3.3]{LL19GlobalUniqueness}.

It is known that $\left.(-\Delta)^{s}u_{f}\right|_{\Omega_{e}}$ is well-defined in $H^{-s}(\Omega_{e})$. Accordingly, the Dirichlet-to-Neumann (DN) map $\Lambda_{f}: C_{c}^{\infty}(\Omega_{e}) \rightarrow H^{-s}(\Omega_{e})$ is defined by $\Lambda_{f}\phi:=\left.(-\Delta)^{s}u_{f}\right|_{\mD}$. Fix any open set $\mD \subset \Omega_{e}$. It was shown in \cite{GRSU20Reconstruction} that one can uniquely determine the potential $f\in C^{0}(\Omega)$ from the knowledge of 
\begin{equation}
G(f):=\Lambda_{f}\phi\equiv\left.(-\Delta)^{s}u_{f}\right|_{\mD}, \label{eq:forward-map-choice} 
\end{equation}
for any fixed $0\not\equiv \phi\in C_{c}^{\infty}(\Omega_{e})$. This result is mainly based on the following antilocality property for fractional Laplacian \cite[Theorem~1.2]{GSU20Calderon}: 
\begin{equation}
\begin{array}{c}
\text{if $w\in H^{r}(\mR^{d})$ for some $r\in\mR$ and there exists $s>0$ such that} \\ 
\text{$w=(-\Delta)^{s}w=0$ in some open set in $\mR^{d}$, then $w=0$ in $\mR^{d}$.} 
\end{array} \label{eq:antilocality}
\end{equation}
Accordingly, the operator $G:C^{0}(\Omega)\rightarrow H^{-s}(\mD)$ is called the \emph{forward operator}. 
By carefully quantifying the arguments, a logarithmic stability is derived in \cite{Rueland2021SingleMeasurement} when $f\in C^{s}(\Omega)$. 
Later in \Cref{lem:qualitative}, we will show that $(-\Delta)^{s}u_{f} \in C^{\infty}(\Omega_{e})$, thereby ensuring that the measurement $G(f)$ is pointwise well-defined. However, it is still impractical to obtain the full measurement $G(f)$ on $\mD$, therefore we develop a probabilistic approach with a discrete set of measurements for this problem.

\subsection{The statistical model}

We write shortly i.i.d. for independent and identically distributed random variables. We randomly choose $N\in\mN$ points from a uniform distribution on $\mD$, that is, for $N\in\mN$, we consider the random variables 
\begin{equation*}
\{X_{i}\}_{i=1}^{N} \overset{\rm i.i.d.}{\sim} \mu ,\quad \mu = \rmd x/\abs{\mD}, 
\end{equation*}
where $\rmd x$ is the usual Lebesgue measure and $\abs{\mD}$ is the Lebesgue measure of $\mD$. Compared to \eqref{eq:forward-map-choice}, it is more realistic to measure 
\begin{equation}
G(f)(x_{i}) = \left((-\Delta)^{s}u_{f}\right)(x_{i}), \label{eq:forward-map}
\end{equation}
where $x_{i}$ are sample points randomly chosen according to the random variable $X_{i}$. We assume the measurement model with a fixed noise level $\sigma>0$ as follows
\begin{equation}
Y_{i} = G(f)(X_{i}) + \sigma W_{i} ,\quad \{W_{i}\}_{i=1}^{N} \overset{\rm i.i.d.}{\sim} \mathcal{N}(0,1), \label{eq:measurement-model}
\end{equation}
where $\mathcal{N}(0,1)$ is the normal distribution with mean 0 and variance 1. We also assume that $W^{(N)}:=\{W_{i}\}_{i=1}^{N}$ and $X^{(N)}:=\{X_{i}\}_{i=1}^{N}$ are independent. We are interested in the inference of $f$ from the observational data $(Y^{(N)},X^{(N)})$ with $Y^{(N)}:=\{Y_{i}\}_{i=1}^{N}$. 

Let $\mO$ be a bounded smooth domain with $\overline{\mO}\subset\Omega$ and let $M_{0} > 1$ be a constant. We consider the following space of parameters 
\begin{equation*}
\mF_{M_{0}}^{\beta}(\mO) := \left\{ f\in H^{\beta}(\Omega): 0 < f < M_{0}, f|_{\Omega\setminus\mO} = 1 , \left. \partial_{\nu}^{j}f \right|_{\partial\mO}=0\, \text{ for all $1\le j\le \beta-1$} \right\}
\end{equation*}
where $\partial_{\nu}$ is the normal derivative in the sense of \cite[Theorem~9.4, Chapter~1]{LM72NonhomogeneousBVPVOL1}. 
Following the approach in \cite{abraham2020statistical,giordano2020consistency}, we consider the link function. 

\begin{assumption}[{\cite[Assumption 2.3]{KW25BayesSubdiffusion}}]\label{assu:link}
Let $\Phi$ be a smooth function which satisfies
\begin{enumerate}
\renewcommand{\labelenumi}{\theenumi}
\renewcommand{\theenumi}{(\roman{enumi})}
\item\label{itm:link-i} $\Phi:(-\infty,\infty)\rightarrow (0,M_{0})$, $\Phi(0)=1$, $\Phi'(z)>0$ for all $z$;
\item\label{itm:link-ii} for any $k\in{\mathbb N}$
\[
\sup_{-\infty<z<\infty}|\Phi^{(k)}(z)|<\infty.
\]
\end{enumerate}
\end{assumption}

Given any link function $\Phi$ as in \Cref{assu:link}, it follows from the argument in \cite{NvdGW20JUQ} that the parameter space can be realized as
\begin{equation*}
\mF_{M_{0}}^{\beta}(\mO) = \{\Phi\circ F: F\in H_{0}^{\beta}(\mO) \}. 
\end{equation*}
Accordingly, we can define the \emph{reparametrized forward map} by 
\begin{subequations}\label{eq:statistical-model}
\begin{equation}
\calG(F):=G(\Phi\circ F) \quad \text{for all $F\in H_{0}^{\beta}(\mO)$,}  \label{eq:reparametrized-forward-map}
\end{equation}
where $G$ is the mapping given in \eqref{eq:forward-map-choice}. Therefore, the model \eqref{eq:measurement-model} can be regarded as a special case of 
\begin{equation}
Y_{i}=\calG(F)(X_{i}) + \sigma W_{i} \quad \text{for $i=1,\cdots,N$.} \label{eq:statistical-model-main}
\end{equation}
\end{subequations} 

The random vectors $(Y_i,X_{i})$ on $\mR\times\mD$ are then i.i.d. with their laws denoted by $\mP_{F}^{i}$ and having the Radon-Nikodym density 
\begin{equation*}
p_{F}(y,x) := \frac{\rmd \mP_{F}^{i}}{\rmd y \times \rmd \mu}(y,x) = \frac{1}{\sqrt{2\pi\sigma^{2}}} \exp \left(-\frac{(y-\calG(F)(x))^{2}}{2\sigma^{2}}\right) \quad \text{for all $y\in\mR,x\in\mD$} 
\end{equation*}
where $\rmd y$ denotes the Lebesgue measure on $\mR$. By a slight abuse of notation, we write $\mP_{F}^{N}=\otimes_{i=1}^{N}\mP_{F}^{i}$ for the joint law of $(Y_{i},X_{i})_{i=1}^{N}$ on $\mR^{N}\times\mD^{N}$, and $\bbE_{F}^{i},\bbE_{F}^{N}$ are the corresponding expectation operators of $\mP_{F}^{i},\mP_{F}^{N}$ respectively. We model the parameter $F\in H_{0}^{\beta}(\mO)$ by a Borel probability measure $\Pi$ supported on the Banach space $C(\mO)$. Since the map $(F,(y,x))\mapsto p_{F}(y,x)$ can be shown to be jointly measurable, the posterior distribution $\Pi(\cdot|Y^{(N)},X^{(N)})$ of $F|Y^{(N)},X^{(N)}$ arising from the model \eqref{eq:statistical-model-main} is given by 
\begin{equation}
\Pi(B|Y^{(N)},X^{(N)}) = \frac{\int_{B}e^{\ell^{N}(F)}\,\rmd\Pi(F)}{\int_{C(\mO)}e^{\ell^{N}(F)}\,\rmd\Pi(F)} \label{eq:posterior}
\end{equation}
for any Borel set $B\subset C(\mO)$, where 
\begin{equation*}
\ell^{N}(F) = - \frac{1}{2\sigma^{2}} \sum_{i=1}^{N} (Y_{i}-\calG(F)(X_{i}))^{2}
\end{equation*}
is the joint log-likelihood function (up to a constant).

In this work, we investigate the consistency of the posterior distribution \eqref{eq:posterior} as the sample size $N$ increases. The aim here is to show that the posterior distribution arising from Gaussian process priors and Gaussian sieve priors concentrates near $f_{0}$ and to derive a bound on the rate of contraction as in \cite{giordano2020consistency}. We will follow the terminologies in \cite{giordano2020consistency} as well as in \cite{gine2021mathematical}. One may also consult the monograph \cite{Nickl23StatisticInverseProblem} for further progress in Bayesian nonlinear statistical inverse problems. 

\subsection{Results with rescaled Gaussian priors}

We begin by introducing the following assumptions. 

\begin{assumption}[{\cite[Condition~3]{giordano2020consistency}}]\label{assu:RKHS}
Let $\alpha > \beta + d/2$, $\beta > 1+d/2$, and $\mathcal{H}$ be a Hilbert space continuously embedded into $H_{0}^{\alpha}(\mO)$. Let $\Pi'$ be a centered Gaussian Borel probability measure on the Banach space $C(\mO)$ that is supported on a separable linear subspace of $H_{0}^{\beta}(\mO)$ and assume that the reproducing kernel Hilbert space (RKHS) of $\Pi'$ equals to $\mathcal{H}$. 
\end{assumption}

An example of $\Pi'$ described in \Cref{assu:RKHS} is given in \cite[Example~2.5]{KW25BayesSubdiffusion}. Now let $\Pi'$ be as in \Cref{assu:RKHS} and $F'\sim \Pi'$, we consider the rescaled prior 
\begin{equation}
\Pi_{N} = \mL(F_{N}) ,\quad F_{N} = \frac{1}{N^{d/(4\alpha+4+2d)}}F'. \label{eq:rescaled-prior}
\end{equation}
Then $\Pi_{N}$ again defines a centered Gaussian prior on $C(\mO)$ and its RKHS is still given by $\mathcal{H}$ but with the rescaled norm
\begin{equation*}
\norm{F}_{\mathcal{H}_{N}} = N^{d/(4\alpha+4+2d)} \norm{F}_{\mathcal{H}} \quad \text{for all $F
\in\mathcal{H}$.} 
\end{equation*} 
We denote the push-forward posterior by 
\begin{equation*}
\tilde{\Pi}_{N}(\cdot|Y^{(N)},X^{(N)}) := \mL(f) ,\quad f=\Phi(F) : F\sim \Pi_{N}(\cdot|Y^{(N)},X^{(N)}). 
\end{equation*}
Our first result shows that the posterior contracts near the ground truth $f_{0}$ in $L^{\infty}$-risk. 

\begin{theorem}\label{thm:1}
Suppose $d\in\mN$. Let $\Omega$, $\mO$ and $\mD$ be nonempty bounded smooth domains in $\mR^{d}$ with $\overline{\mO}\subset\Omega$ and $\overline{\Omega}\cap\overline{\mD}=\emptyset$. Let $0<s<1$, $0\not\equiv\phi\in C_{c}^{\infty}(\Omega_{e})$, and $\Phi$ be the link function given in \Cref{assu:link}. 
Let $(\mathcal{H},\Pi')$ satisfy \Cref{assu:RKHS}, $\Pi_{N}$ be the rescaled prior given in \eqref{eq:rescaled-prior} and let the ground truth $f_{0}$ be given by $f_{0}=\Phi\circ F_{0}$, where $F_{0}\in H^{\alpha}(\mO)$ with $\supp\,(F_{0})\subset\mO$. Let $\Pi_{N}(\cdot|Y^{(N)},X^{(N)})$ be generated through the model \eqref{eq:statistical-model} of law $\mP_{F_{0}}^{N}$. Then for any $D>0$, there exist constants $C>0$ and $\mu>0$ such that 
\begin{equation}\label{eq_statement_2.3}
\tilde{\Pi}_{N}\left( f : \norm{f-f_{0}}_{L^{\infty}(\mO)} > C(\log N)^{-\mu} | Y^{(N)},X^{(N)} \right) = O_{\mP_{F_{0}}^{N}} (e^{-DN\delta_{N}^{2}}) 
\end{equation}
as $N\rightarrow\infty$, where $\delta_{N}=N^{-\frac{\alpha}{2\alpha+d}}$. 
\end{theorem}

\begin{remark}
    Notice that in \eqref{eq_statement_2.3}, we actually have $\norm{f-f_{0}}_{L^{\infty}(\mO)}=\norm{f-f_{0}}_{L^{\infty}(\Om)}$. This is true since $f_0=\Phi\circ F_0$, $f = \Phi\circ F$, with $F_0\in H_{0}^{\alpha}(\mO)$, $F\in\mathcal{H}$. When we combine this with the properties of the link function from \Cref{assu:link}, we get that $f_0,f=1$ in $\Om\setminus\mO$. 
\end{remark}

To obtain an estimator of the unknown true coefficient $f_{0}$, in view of the link function $\Phi$, it is often convenient to derive an estimator of $F_{0}$. One natural choice of such estimator is the posterior mean $\overline{F}_{N} := \bbE^{\Pi}(F|Y^{(N)},X^{(N)})$ of $\Pi_{N}(\cdot|Y^{(N)},X^{(N)})$, which can be approximated numerically by an MCMC algorithm, see \Cref{sec:MCMC}. Our second result establishes a contraction rate for the convergence $\overline{F}_{N}$ to $F_{0}$. 

\begin{theorem}\label{thm:2}
Assume that the hypotheses of \Cref{thm:1} hold. Then there exist constants $\tilde{C}>0$ and $\mu>0$ such that 
\begin{equation*}
\mP_{F_{0}}^{N} \left( \norm{\overline{F}_{N}-F_{0}}_{L^{\infty}(\mO)} > \tilde{C}(\log N)^{-\mu} \right) \rightarrow 0 \quad \text{as $N\rightarrow\infty$.}
\end{equation*}
As an immediate consequence, there exist constants $\tilde{C}'>0$ and $\mu>0$ such that 
\begin{equation*}
\mP_{f_{0}}^{N} \left( \norm{\Phi(\overline{F}_{N})-f_{0}}_{L^{\infty}(\mO)} > \tilde{C}'(\log N)^{-\mu} \right) \rightarrow 0 \quad \text{as $N\rightarrow\infty$.}
\end{equation*}
\end{theorem}

\subsection{Results with Gaussian sieve priors}

From a computational perspective, it is useful to consider sieve priors that are finite-dimensional approximations of the function space supporting the prior. Here, we will use a truncated Karhunen--Lo{\`e}ve type expansion in terms of the Daubechies wavelets considered in  \cite[Appendix~B]{giordano2020consistency} or \cite[Chapter~4]{gine2021mathematical}. Let $\{\Psi_{\ell r}: \ell \ge -1, r\in \mZ^d  \}$ be the collection of ($d$-dimensional) compactly supported Daubechies wavelets\footnote{This can be easily constructed from the 1-dimensional Daubechies wavelets as in \cite[Theorem~4.2.10]{gine2021mathematical}, and the scaling functions (in the sense of \cite[Definition~4.2.1]{gine2021mathematical}) are interpreted as the `first' wavelet due to the wavelet series expansion \cite[(4.32)]{gine2021mathematical}.}, which form an orthonormal basis of $L^2(\mR^d)$. Let $\mathcal{K}$ be a compact set in $\mO$, and for each integer $\ell \ge -1$, let ${\mathcal R}_{\ell} = \left\{ r \in \mZ^{d} : \supp\,(\Psi_{\ell r}) \cap \mathcal{K} \neq \emptyset \right\}$. Let $\mathcal{K}'$ be another compact subset in $\mO$ such that $\mathcal{K} \subsetneq \mathcal{K}'$, and let $\chi \in C_{c}^{\infty}(\mO)$ be a cut-off function with $\chi=1$ on $\mathcal{K}'$. For any real $t > d/2$, we consider the prior (introduced in \cite[Section~2.2.3 and Remark~26]{giordano2020consistency}) 
\begin{equation}
\Pi_{j}'={\mathcal L}(\chi F_{j}),\quad F_{j} 
=\sum_{\begin{subarray}{c} -1\le\ell \le j \\ r\in{\mathcal R}_{\ell} \end{subarray}} 2^{-\ell t}
F_{\ell r}\Psi_{\ell r},\;\; F_{\ell r}\overset{\rm iid}{\sim}{\mathcal N}(0,1), \label{eq:sieve-prior}
\end{equation}
where $j\in\mN$ is a truncation level. According to \cite[Exercise~2.6.5]{gine2021mathematical}, for each $j\in\mN$, the prior $\Pi_{j}'$ above defines a centered Gaussian probability measure supported in the finite dimensional subspace 
\begin{equation*}
\mathcal{H}_{j} := {\rm span}\, \{ \chi \Psi_{\ell r} : -1 \le \ell \le j , r \in \mathcal{R}_{\ell} \},
\end{equation*}
with the RKHS norm given in \cite[(B2)]{giordano2020consistency}. 

\begin{theorem}\label{thm:3} 
Consider the rescaled prior $\Pi_{N}$ defined in \eqref{eq:rescaled-prior} with $F'\sim\Pi_{j(N)}'$ such that  $2^{j(N)}\simeq N^{1/(2\alpha+d)}$. Then the conclusions in \Cref{thm:1} and \Cref{thm:2} remain valid. 
\end{theorem}

The proof of \Cref{thm:3} only requires minor modifications from the proof of \Cref{thm:1} and \Cref{thm:2}, and all necessary modifications are listed in \cite[Section~3.2]{giordano2020consistency}. Therefore, we omit the details. 

\subsection{Tightness of the logarithmic contraction rate}

The next theorem shows that the logarithmic contraction rate (in \Cref{thm:1,thm:2,thm:3}) is optimal in the statistical minimax sense. We show this in a special case of $\Om=\{x\in\mR^d : \abs{x}<1\}$ and $\mD=\{x\in\mR^d : 2\leq\abs{x}<3\}$. This setup originates from an instability result for the Calder\'{o}n problem for the fractional Schr\"{o}dinger equation proved in \cite{RS18Instability}.

\begin{theorem}\label{thm_minimax_optimal}
    Let $d\in\mN$, $\beta>1+\frac{d}{2}$ be an integer, $M_0>1$, $\Om=\{x\in\mR^d : \abs{x}<1\}$, $\mO$ a smooth nonempty bounded domain such that $\overline{\mO}\subset\Om$ and $\mD=\{x\in\mR^d : 2\leq\abs{x}<3\}$. Then for any $\mu>2+\frac{3}{d}$ 
    \begin{equation*}
        \inf_{\tilde f}\sup_{f\in\mF_{M_{0}}^{\beta}}\mP^N_f\left(\norm{\tilde f-f}_{L^{\infty}(\mO)}\geq\frac{1}{2}(\log N)^{-\mu}\right)\geq \frac14
    \end{equation*}
    where the infimum is taken over all estimators $\tilde f=\tilde f(Y)$, namely measurable functions, of the data $Y\sim\mP^N_f$.
\end{theorem}

\subsection{Remarks on the Bernstein-von Mises theorem} 

Consistency alone does not suffice for uncertainty quantification. A central goal would be to establish a Bernstein-von Mises (BvM) theorem characterizing the asymptotic shape of the posterior distribution. Although this theorem generally fails in infinite-dimensional settings (see \cite[Remark~9]{nickl2020bernstein}), progress has been made for the Calder{\'o}n problem with piecewise-constant conductivities, corresponding to a parametric setting \cite{Boh2023BvM}. Further results in semiparametric settings can be found in \cite{nickl2020bernstein,Nickl2024BvMtime}.

\section{Estimates for the fractional Schr\"{o}dinger equation}\label{sec_estimates}

Before proving our statistical results, we first derive several key estimates for the deterministic model. 

\subsection{Estimates for the forward problem}\label{sec_forward_estimates}

The following lemma establishes that $(-\Delta)^{s}u_{f} \in C^{\infty}(\Omega_{e})$ for $\phi\in C_{c}^{\infty}(\Omega_{e})$. It can be proved using an argument similar to the one in \cite[Theorem~1.3.16]{PSU_book}. Let us recall that the singular support of a distribution $u \in \mathcal{D}'(\R^d)$ is defined as a complement of the largest set where $u$ is smooth. In particular, $$\ssupp(u) := \R^d \setminus \{\, x \in \R^d \,:\,\text{$u$ is smooth near $x$}\}.$$

\begin{lemma}\label{lem:qualitative}
Let $d\in\mN$, $0<s<1$ and $\Omega\subset\mR^{d}$ be a nonempty bounded smooth domain. Let $\phi\in C_{c}^{\infty}(\Omega_{e})$, $f\in L^{\infty}(\Omega)$ and $u=u_{f}\in H^{s}(\mR^{d})$ be the unique solution of \eqref{eq:Dirichlet-problem}. Then 
\begin{equation*}
\ssupp((-\Delta)^s u) \cap \Omega_e = \emptyset. 
\end{equation*}
\end{lemma}

\begin{proof}
Let $\chi\in C^{\infty}_c(\mR^d)$ satisfy $\chi(\xi)=1$ for $\abs{\xi}\leq\frac{1}{2}$ and $\chi(\xi)=0$ for $\abs{\xi}\geq1$. Then
\begin{align*}
& (-\Delta)^{s} u = \mathcal{F}^{-1}\left(\abs{\xi}^{2s}(\mathcal{F} u)(\xi)\right)\\
& \quad = \mathcal{F}^{-1}\left((1-\chi(\xi))\abs{\xi}^{2s}(\mathcal{F} u)(\xi)\right) + \mathcal{F}^{-1}\left(\chi(\xi)\abs{\xi}^{2s}(\mathcal{F} u)(\xi)\right)\\
& \quad = Eu + Su.
\end{align*}
Since $u\in H^s(\mR^d)$ and $\chi$ has compact support, we know that $\chi(\xi)\abs{\xi}^{2s}(\mathcal{F} u)(\xi)\in L^1$ with compact support. Hence, $Su$ is smooth by the properties of the Fourier transform. In addition, $E$ is a classical elliptic pseudodifferential operator of order $2s$ since its symbol $(1-\chi(\xi))\abs{\xi}^{2s}$ is a classical elliptic symbol of order $2s$. Hence,
\begin{equation*}
\ssupp(Eu)\cap\Omega_e\subset \ssupp(u)\cap\Omega_e =  \ssupp(\phi)\cap\Omega_e,
\end{equation*}
which then concludes the proof. 
\end{proof}

\begin{remark} In fact, it holds that $\ssupp \Lambda_f \phi \subset \ssupp \phi \cap \mD$ whenever $\phi \in H^s(\R^d)$.
\end{remark}
Now as we know that our pointwise measurement model is well-defined, we can start working towards proving our main theorems. Our aim is to use the machinery developed in \cite{giordano2020consistency}. For that, we need certain forward estimates for the reparametrized forward map $\calG$. These estimates follow from similar estimates for the forward map $G$.  Our first intermediate result is a local Lipschitz estimate for the forward problem.

\begin{lemma}\label{lem:forward-estimate}
Let $d\in\mN$, $0<s<1$, $\Omega$ and $\mD$ be nonempty bounded smooth domains in $\mR^{d}$ with $\overline{\mD}\cap\overline{\Omega}=\emptyset$. Let $\phi\in C^{\infty}_c(\mD)$ and let $f_1, f_2 \in L^{\infty}(\Om)$ be such that $f_1, f_2 \geq0$. Then there exists a constant $C=C(d,s,\Omega,\abs{\mD},\dist(\Omega,\mD),\phi)>0$,
independent of both $f_{1}$ and $f_{2}$, such that
\begin{equation}\label{forward_estimate}
\norm{G(f_{1})-G(f_{2})}_{L^2(\mD)}\leq C\norm{f_1-f_2}_{L^2(\Om)}, 
\end{equation}
where $G$ is the forward map given in \eqref{eq:forward-map}. 
\end{lemma}

We provide the proof in \Cref{sec_appendix}. To conclude this section, we prove a uniform bound for the forward map $G$. As before, we get a similar property for $\calG$ from this result.

\begin{lemma}\label{lem:uniform-bound}
Let $d\in\mN$, $0<s<1$, $\Omega$ and $\mD$ be nonempty bounded smooth domains in $\mR^{d}$ with $\overline{\mD}\cap\overline{\Omega}=\emptyset$. Let $\phi\in C^{\infty}_c(\mD)$, $f \in L^{\infty}(\Om)$ be such that $f \geq0$ and let $u_{f}$ be the unique solution to \eqref{eq:Dirichlet-problem}. Then there exists a constant $C=C(d,s,\abs{\Omega},\dist(\Omega,\mD),\phi)>0$, independent of $f$, such that 
\begin{equation*}
\norm{G(f)}_{L^{\infty}(\mD)} \le C. 
\end{equation*}
\end{lemma}
Again, we provide the proof in \Cref{sec_appendix}.

\subsection{Stability estimates for the inverse problem}

One of the key ingredients to prove the statistical contraction rates is a stability result for the inverse problem at hand. 
To suit our purposes, we need to refine the stability estimate established in \cite[Theorem~1]{Rueland2021SingleMeasurement}, making explicit the expressions for $\norm{f_{1}}_{C^{s}(\Omega)}$ and $\norm{f_{2}}_{C^{s}(\Omega)}$.

\begin{proposition}\label{prop:inverse-estimate}
Let $d\in\mN$, let $\Omega$, $\mO$ and $\mD$ be nonempty bounded smooth domains in $\mR^{d}$ with $\overline{\mO}\subset\Omega$ and $\overline{\Omega}\cap\overline{\mD}=\emptyset$. Let $f_{1},f_{2}\in C^{s}(\Omega)$, $f_1,f_2\geq0$ with 
\begin{equation}
\norm{f_{1}}_{L^{\infty}(\Omega)}\vee\norm{f_{2}}_{L^{\infty}(\Omega)} \le M_{0} \label{eq:apriori}
\end{equation}
for some constant $M_{0}>0$. Then there exist constants $C=C(\Omega,\mD,\mO,d,s,M_{0},\phi)>0$ and $\gamma=\gamma(\Omega,\mO,d,s,M_{0},\phi)>0$ such that 
\begin{equation}
\norm{f_{1}-f_{2}}_{L^{\infty}(\mO)} \le C\left(1+\norm{f_{1}}_{C^{s}(\Omega)}^{\gamma} \vee \norm{f_{2}}_{C^{s}(\Omega)}^{\gamma}\right) \omega\left(\norm{G(f_{1})-G(f_{2})}_{H^{-s}(\mD)}\right),  \label{eq:stability-quantify1}
\end{equation}
where 
\begin{equation*}
\omega(t) = \left\{\begin{aligned}
& 0 && \text{if $t=0$} \\ 
& |\log(\tilde{E}t^{-1})|^{-\mu} && \text{if $0<t<\tilde{E}/2$,}  \\
& \abs{\log 2}^{-\mu} && \text{if $t\ge \tilde{E}/2$.}
\end{aligned}\right.
\end{equation*}
for some $\mu=\mu(\Omega,\mD,\mO,d,s,M_{0},\phi)>0$ and $\tilde{E}=\tilde{E}(\Omega,\mD,d,s,M_{0},\phi)>0$. 
\end{proposition}

It is important to note that the logarithmic modulus of continuity $\omega$ is independent of both $f_1$ and $f_2$. The proof is presented in \Cref{sec_appendix}.

\section{Statistical conclusions}\label{sec_proof_of_main}

After establishing several estimates for the forward map $G$ and for the deterministic inverse problem, we are now in a position to derive our main statistical conclusions.

\subsection{Posterior contraction rates}  

We first verify the required forward estimates for the reparametrized map $\calG$. These estimates will be needed in the proofs of our main theorems on the Bayesian approach. 

\begin{proposition}\label{prop:sufficient} 
Let $d\in\mN$ and $\Omega$, $\mO$, $\mD$ be nonempty bounded smooth domains in $\mR^{d}$ with $\overline{\mO}\subset\Omega$ and $\overline{\Omega}\cap\overline{\mD}=\emptyset$. Let $0<s<1$, $0\not\equiv\phi\in C_{c}^{\infty}(\Omega_{e})$ and let $\Phi$ be the link function given in \Cref{assu:link}. Let $\beta > 1 + d/2$ be an integer and $\calG$ be the reparametrized forward map given by \eqref{eq:reparametrized-forward-map}. Then one has  
\begin{equation*}
\norm{\calG(F_{1})-\calG(F_{2})}_{L^{2}(\mD)} \lesssim \norm{F_{1}-F_{2}}_{L^{2}(\mO)} \quad \text{for all $F_{1},F_{2}\in H_{0}^{\beta}(\mO)$}
\end{equation*}
and 
\begin{equation*}
\norm{\calG}_{H_{0}^{\beta}(\mO)\rightarrow L^{\infty}(\mD)} \equiv \sup_{F\in H_{0}^{\beta}(\mO)} \norm{\calG(F)}_{L^{\infty}(\mD)} < \infty. 
\end{equation*}
\end{proposition}

\begin{proof}
This proposition is an immediate consequence of \Cref{lem:forward-estimate} and \Cref{lem:uniform-bound}. 
\end{proof}

We are now ready to prove the main results. 

\begin{proof}[Proof of \Cref{thm:1}]
First of all, it follows from \Cref{prop:sufficient} that the assumptions in \cite[Theorem~14]{giordano2020consistency}) are satisfied. More precisely, the conditions in \cite[Lemma~5.1]{KW25BayesSubdiffusion}, a formulation of the aforementioned theorem, are satisfied with $\tau=1$ and $\kappa=0$. Hence, there exists a sufficiently large $L>0$ such that 
\begin{equation*}
\Pi_{N}\left( F:\norm{\calG(F)-\calG(F_{0})}_{L^{2}(\mD)} > L\delta_{N} | Y^{(N)} , X^{(N)} \right) = O_{\mP_{F_{0}}^{N}} (e^{-DN\delta_{N}^{2}}) \quad \text{as $N\rightarrow\infty$,} 
\end{equation*}
where $\delta_{N}=N^{-\frac{\alpha}{2\alpha+d}}$. Furthermore, there exists a sufficiently large $M$ such that 
\begin{equation}
\Pi_{N}\left(F:\norm{F}_{C^{1}(\mO)} > M | Y^{(N)} , X^{(N)}\right)= O_{\mP_{F_{0}}^{N}} (e^{-DN\delta_{N}^{2}}) \quad \text{as $N\rightarrow\infty$.} \label{eq:C1}
\end{equation}
From \Cref{prop:inverse-estimate}, we now see that 
\begin{equation*}
\begin{aligned}
& \tilde{\Pi}_{N}\left( f : \norm{f-f_{0}}_{L^{\infty}(\mO)} > C\omega(L'\delta_{N}) |Y^{(N)} , X^{(N)} \right) \\ 
& \quad \le \tilde{\Pi}_{N}\left( f : \norm{G(f)-G(f_{0})}_{L^{2}(\mD)} > L\delta_{N} |Y^{(N)} , X^{(N)} \right) \\
& \qquad + \tilde{\Pi}_{N}\left( f : \norm{f}_{C^{s}(\mO)} > M' |Y^{(N)} , X^{(N)} \right) \\ 
& \quad \le \Pi_{N}\left( F : \norm{\calG(F)-\calG(F_{0})}_{L^{2}(\mD)} > L\delta_{N} |Y^{(N)} , X^{(N)} \right) \\
& \qquad + \Pi_{N}\left( F : \norm{F}_{C^{1}(\mO)} > M |Y^{(N)} , X^{(N)} \right) \\ 
& \quad = O_{\mP_{F_{0}}^{N}} (e^{-DN\delta_{N}^{2}}) \quad \text{as $N\rightarrow\infty$.}
\end{aligned}
\end{equation*}
Finally, we observe that 
\begin{equation*}
\begin{aligned}
& C\omega(L'\delta_{N}) = C\abs{\log(\tilde{E}L'N^{-\frac{\alpha}{2\alpha+d}})}^{-\mu} \\ 
& \quad = C\left| \log (\tilde{E}L') - \frac{\alpha}{2\alpha+d}\log N \right|^{-\mu} \sim C(\log N)^{-\mu} \quad \text{as $N\rightarrow\infty$.}
\end{aligned}
\end{equation*}
Hence, we have concluded our theorem. 
\end{proof}

\begin{proof}[Proof of \Cref{thm:2}]
The idea of this proof is to use tools from the general machinery developed in \cite{giordano2020consistency, gine2021mathematical}. In particular, the proof is based on \cite[Theorem~6]{giordano2020consistency} (almost identical to \cite[Theorem~2.6]{FKW24ProbabilityInverseScattering}). By Jensen's inequality, it is enough to show that
\begin{equation*}
\mP_{F_{0}}^{N} \left( \bbE\left[\norm{\overline{F}_{N}-F_{0}}_{L^{\infty}(\mO)}\right] > \tilde{C}(\log N)^{-\mu} \right) \rightarrow 0 \quad \text{as $N\rightarrow\infty$.}
\end{equation*}
Next, one can split the expectation to two parts: For a large $M>0$, we look at the case with $\norm{\overline{F}_N}_{C^{1}(\mO)}\leq M$ and the one with $\norm{\overline{F}_N}_{C^{1}(\mO)}> M$.
For $\norm{\overline{F}_N}_{C^{1}(\mO)}> M$, one uses the Cauchy-Schwarz inequality, \Cref{thm:1}, \cite[Lemma 16, Lemma 23]{giordano2020consistency} and \cite[Lemma 7.3.2]{gine2021mathematical} to get the desired asymptotic behavior. For $\norm{F}_{C^{1}(\mO)}\leq M$, one reduces to the first case by using the properties of the link function together with the mean value theorem and the inverse function theorem. For further details, see the proofs of \cite[Theorem~6]{giordano2020consistency}. 
\end{proof}

\subsection{Minimax optimality of the contraction rates} 

For the proof of \Cref{thm_minimax_optimal}, we will follow the method used in the proof \cite[Theorem 2]{abraham2020statistical} (see also the proof of \cite[Theorem 2.8]{FKW24ProbabilityInverseScattering}).
We shall use the exponential instability of the Calder{\'o}n problem for the fractional Schr\"{o}dinger equation proved in \cite[Corollary 1.2]{RS18Instability}. We will recall the result here for the convenience of the reader (adopted to our notation).

\begin{theorem}[\text{\cite[Corollary 1.2]{RS18Instability}}]\label{thm_instability}
    Let $d\in\mN$ and $B_1\subset\mR^d$. Assume that $m\in\mN$. There is a potential $f\in L^{\infty}(B_1)$, a sequence of errors $\{\e_k\}_{k\in\mN}$ with $\e_k\to0$, and a sequence of potentials $\{f_k\}_{k\in\mN}\subset C^m(\overline{B}_1)$ such that 
    \begin{enumerate}[label=(\roman*)]
        \item $\norm{f-f_k}_{L^{\infty}(B_1)}\geq\frac{\e_k}{2}$ and $\norm{\Lambda_f-\Lambda_{f_k}}_{L^2(B_3\setminus\overline{B}_2)\to L^2(B_3\setminus\overline{B}_2)}\leq C\exp\left(-\e_k^{-\frac{d}{(2d+3)m}}\right),$
        \item for all $0\leq m'< m$ we have $f_k\to f$ in $C^{m'}(\overline{B}_1)$.
    \end{enumerate}
\end{theorem}

We now prove \Cref{thm_minimax_optimal}.

\begin{proof}[Proof of \Cref{thm_minimax_optimal}]
    Following \cite[Section 6.3.1 and Theorem 6.3.2]{gine2021mathematical}, we can reduce the proof of minimax lower bound to a testing problem with two hypotheses. By using the test $\Psi=\one_{\{\norm{\tilde f-f_1}_{L^{\infty}}<\norm{\tilde f-f_0}_{L^{\infty}}\}}$, we have $H_0\colon f=f_0$ against $H_1\colon f=f_1$. Now from the reduction made in the beginning of \cite[Section 6.3.1]{gine2021mathematical}, we obtain 
    \begin{equation*}
        \inf_{\tilde f}\sup_{f\in\mF_{M_{0}}^{\beta}}\mP^N_f\left(\norm{\tilde f-f}_{L^{\infty}(\mO)}\geq\frac{1}{2}(\log N)^{-\mu}\right)\geq\inf_{\psi}\max\{\mP^N_{f_0}(\psi\neq0), \mP^N_{f_1}(\psi\neq1)\}.
    \end{equation*}
    Recall that $p_{f}^{\otimes N}$ is the Radon-Nikodym density of $\mP^N_f=\otimes_{i=1}^{N}\mP^i_f$. Next, let us define an event $B=\left\{\frac{p^{\otimes N}_{f_0}}{p^{\otimes N}_{f_1}}\geq\frac{1}{2}\right\}$ for which we have
    \begin{align*}
        \mP^N_{f_0}(\psi\neq0)&=\int_{\{\psi\neq0\}}p^{\otimes N}_{f_0}\rmd\mu \geq \int\one_{\{\psi=1\}}\one_B p^{\otimes N}_{f_0}\rmd\mu = \int\one_{\{\psi=1\}}\one_B \frac{p^{\otimes N}_{f_0}}{p^{\otimes N}_{f_1}}p^{\otimes N}_{f_1}\rmd\mu\\
        &\geq \frac{1}{2}\mP^N_{f_1}(\psi=1)>\frac{1}{2}\left(\mP^N_{f_1}(\psi=1)-\mP^N_{f_1}(B^c)\right).
    \end{align*}
    Then, denoting $p_1=\mP^N_{f_1}(\psi=1)$ gives
    \begin{align*}
        \max\{\mP^N_{f_0}(\psi\neq0), \mP^N_{f_1}(\psi\neq1)\} &= \max\left\{\frac{1}{2}(p_1 - \mP^N_{f_1}(B^c), 1 - p_1\right\}\\
        &\geq \inf_{p\in[0,1]}\max\left\{\frac{1}{2}(p - \mP^N_{f_1}(B^c), 1 - p\right\}.
    \end{align*}
    This infimum is obtained when $1-p=\frac{1}{2}(p - \mP^N_{f_1}(B^c)$ and, thus, $p=1-\frac{1}{3}\mP^N_{f_1}(B)$. Hence,
    \begin{equation*}
        \max\{\mP^N_{f_0}(\psi\neq0), \mP^N_{f_1}(\psi\neq1)\}\geq \frac{1}{3}\mP^N_{f_1}(B).
    \end{equation*}
    Estimating the term $\mP^N_{f_1}(B)$ further, we have
    \begin{align*}
        \mP^N_{f_1}(B) &= \mP^N_{f_1}\left(\frac{p^{\otimes N}_{f_1}}{p^{\otimes N}_{f_0}}\leq2\right) = 1 - \mP^N_{f_1}\left(\log\frac{p^{\otimes N}_{f_1}}{p^{\otimes N}_{f_0}}>\log2\right)\\
        &\geq 1 - \mP^N_{f_1}\left(\left|\log\frac{p^{\otimes N}_{f_1}}{p^{\otimes N}_{f_0}} \right| >\log2\right)\geq 1 - \frac{1}{\log2}\bbE^N_{f_1}\left[\left| \log\frac{p^{\otimes N}_{f_1}}{p^{\otimes N}_{f_0}} \right| \right],
    \end{align*}
    where in the last estimate we used Markov's inequality. Using the second Pinsker inequality (see e.g. \cite[Proposition 6.1.7 (b)]{gine2021mathematical}) gives
    \begin{equation*}
        \mP^N_{f_1}(B)\geq 1-\frac{1}{\log2}\left(\KL(p_{f_0}^{\otimes N},p_{f_1}^{\otimes N}) + \sqrt{2\KL(p_{f_0}^{\otimes N},p_{f_1}^{\otimes N})}\right),
    \end{equation*}
    where $\KL(\,\cdot\,,\,\cdot\,)$ is the Kullback-Leibler distance defined by
    \begin{equation*}
        \KL(p_{f_0}^{\otimes N},p_{f_1}^{\otimes N}) = \bbE^{N}_{f_0}\left(\log\frac{p_{f_0}^{\otimes N}}{p_{f_1}^{\otimes N}}\right).
    \end{equation*}
    Putting the above together, we have proven so far that
    \begin{equation}
    \begin{aligned}
    & \inf_{\tilde f}\sup_{f\in\mF_{M_{0}}^{\beta}}\mP^N_f\left(\norm{\tilde f-f}_{L^{\infty}(\mO)}\geq\frac{1}{2}(\log N)^{-\mu}\right) \\ 
    & \quad \geq \frac13-\frac{1}{3\log2}\left(\KL(p_{f_0}^{\otimes N},p_{f_1}^{\otimes N}) + \sqrt{2\KL(p_{f_0}^{\otimes N},p_{f_1}^{\otimes N})}\right).
    \end{aligned}\label{eq_infsup_proof} 
    \end{equation}
    Let us then focus on the Kullback-Leibler distance. We have
    \begin{equation*}
        \KL(p_{f_0}^{\otimes N},p_{f_1}^{\otimes N}) = N\bbE^1_{f_0}\left[\log\frac{p_{f_0}(Y_1,X_1)}{p_{f_1}(Y_1,X_1)}\right] = \frac{N}{2\sigma^2}\norm{G(f_0) - G(f_1)}^2_{L^2(\mD)},
    \end{equation*}
    since $(X_i, Y_i)$ are independent in the first equality, and the second equality comes from \cite[Lemma 23]{giordano2020consistency}. Recall that $G$ is the DN map (see \eqref{eq:forward-map}) and, thus, we have
    \begin{equation*}
        \norm{G(f_0) - G(f_1)}_{L^2(\mD)}\leq \norm{\Lambda_{f_0}-\Lambda_{f_1}}^{1/2}_{L^2(\mD)\to L^2(\mD)}
    \end{equation*}
    (see, for example, the beginning of the proof of \Cref{lem:forward-estimate}). Let $\mu>\frac{2d+3}{d}$ and $\delta = (\log N)^{-\mu}$. Note that $N$ will be taken to be large at the end of the proof. By \Cref{thm_instability}, there are $f_0, f_1$ such that 
    \begin{equation*}
        \norm{f_0-f_1}_{L^{\infty}(\Om)}\geq \frac{\delta}{2}\quad\text{and}\quad \norm{\Lambda_{f_0}-\Lambda_{f_1}}_{L^2(\mD)\to L^2(\mD)}\leq C\exp\left(-\delta^{-\frac{d}{2d+3}}\right).
    \end{equation*}
    Hence,
    \begin{align*}
        \KL(p_{f_0}^{\otimes N},p_{f_1}^{\otimes N}) \leq \frac{N}{2\sigma^2}C\exp\left(-\delta^{-\frac{d}{2d+3}}\right) = \frac{C}{2\sigma^2}\exp\left(\log N - (\log N)^{\frac{\mu d}{2d+3}}\right)
    \end{align*}
    and, for the choice of $\mu$, the exponent of $\exp$ is negative. Thus, by choosing $N$ sufficiently large, we obtain
    \begin{equation*}
        \KL(p_{f_0}^{\otimes N},p_{f_1}^{\otimes N}) \leq \tilde\e
    \end{equation*}
    for some $\tilde\e>0$. Using this in \eqref{eq_infsup_proof} gives
    \begin{equation*}
        \inf_{\tilde f}\sup_{f\in\mF_{M_{0}}^{\beta}}\mP^N_f\left(\norm{\tilde f-f}_{L^{\infty}(\mO)}\geq\frac{1}{2}(\log N)^{-\mu}\right)\geq \frac14.
    \end{equation*}
    Choosing $N$ large enough completes the proof.
\end{proof}

\section{Numerical demonstrations of the Bayesian procedure\label{sec:MCMC}}

The main objective of this section is to numerically illustrate the convergence of the Bayesian procedure using a Markov Chain Monte Carlo (MCMC) algorithm based on the log-likelihood. We remark that there exist faster and more accurate steepest-descent methods, based on a suitable Tikhonov functional, which are commonly employed in practice \cite{dwivedi2026galerkinfiniteelementmethod,Li24NumericalFractionalCalderon2,Li24NumericalFractionalCalderon1}. Due to the computational complexity, we restrict our consideration to the one-dimensional case $d=1$. Let $\ell$ be a positive real number, $K$ be a positive integer and set the mesh size $h = 2\ell/K$. Based on \cite[(2.10)]{DvWZ18FDMFractionalLaplacian} (with an appropriately chosen parameter), the numerical approximation of $(-\Delta)^{s}v(x)$ for $x \in (-\ell,\ell)$, under the assumption that $v(x)=0$ outside $(-\ell,\ell)$, is given by $(-\Delta)_{h}^{s}v(x_{i})=A\mathbf{v}$, where $x_{i}=-\ell+ih$ are grid points for $0\le i \le K$, $\mathbf{v}=(v(x_{1}),v(x_{2}),\cdots,v(x_{K-1}))^{\intercal}$ is a column vector and the matrix $A=(A_{ij})$ is a symmetric Toeplitz matrix given by 
\begin{equation*}
A_{ij} = \frac{c_{1,2s}}{(1-s)h^{2s}} \left\{ \begin{aligned}
& \begin{aligned}
& \sum_{k=2}^{K} \frac{(k+1)^{1-s}-(k-1)^{1-s}}{k^{1+s}} \\ 
& \quad + \frac{K^{1-s}-(K-1)^{1-s}}{K^{1+s}} + 2^{1-s} + \frac{1-s}{s K^{2s}}
\end{aligned}  && \text{when $j=i$} \\ 
& - \frac{(\abs{j-i}+1)^{1-s}-(\abs{j-i}-1)^{1-s}}{2\abs{j-i}^{1+s}} && \text{when $j \neq i$ and $j \neq i\pm1$,} \\ 
& -2^{-s} && \text{when $j=i\pm1$,}
\end{aligned} \right.
\end{equation*}
for all $i,j=1,2,\cdots,K-1$, where $c_{1,2s}=\frac{4^{s}s\Gamma(\frac{1+2s}{2})}{\pi^{\frac{1}{2}}\Gamma(1-s)} > 0$. We also refer to \cite{DZ19FDMFractionalLaplacian} for discretizations of $(-\Delta)^{s}$ in the two- and three-dimensional cases. 

For any given $f$, we aim to compute the solution $u_f$ of \eqref{eq:Dirichlet-problem} numerically. To this end, we define $v_f := u_f - \phi$, which satisfies 
\begin{equation}
((-\Delta)^{s}+f)v_{f} = F:=-(-\Delta)^{s}\phi \text{ in $\Omega$} ,\quad v_{f}|_{\Omega_{e}}=0. \label{eq:Dirichlet-problem-transformed}  
\end{equation}
Assuming $\supp(\phi) \subset (-3,-2)$, we take $\Omega = (-1,1)$, $\mD=(-3,-1)\cup(1,3)$ and $\ell = 3$. For convenience, we choose $K = 6M$ with $M \in \mN$, and define the grid points 
\begin{equation*}
x_{i}=-3+ih = -3+i\frac{6}{K} = -3+i\frac{1}{M} 
\end{equation*}
for $0\le i \le K=6M$ as above. It is easy to see that
\begin{equation*} 
\begin{aligned} 
& \text{$x_{i}\in\Omega=(-1,1)$ if and only if $2M<i<4M$,} \\ 
& \text{$x_{i}<-2$ if and only if $i<M$,} \\ 
& \text{$x_{i}>2$ if and only if $i>5M$,}
\end{aligned} 
\end{equation*} 
which facilitates a straightforward discretization of \eqref{eq:Dirichlet-problem-transformed} and allows for a simple numerical computation of the forward map \eqref{eq:forward-map}. 
In our simulation, we choose $M=50$, $s=\frac{1}{2}$, $\mO=(-\frac{1}{2},\frac{1}{2})$ and 
\begin{equation*}
\phi(x) = 10000\exp\left(\frac{1}{(x+\frac{5}{2})^{2}-\frac{1}{4}}\right) \chi_{(-3,-2)}(x) \quad \text{for all $x\in\mR$.}
\end{equation*} 
See \Cref{fig:examples} for some examples demonstrating the forward problem. 

\begin{figure}
\centering
\includegraphics[width=0.5\textwidth]{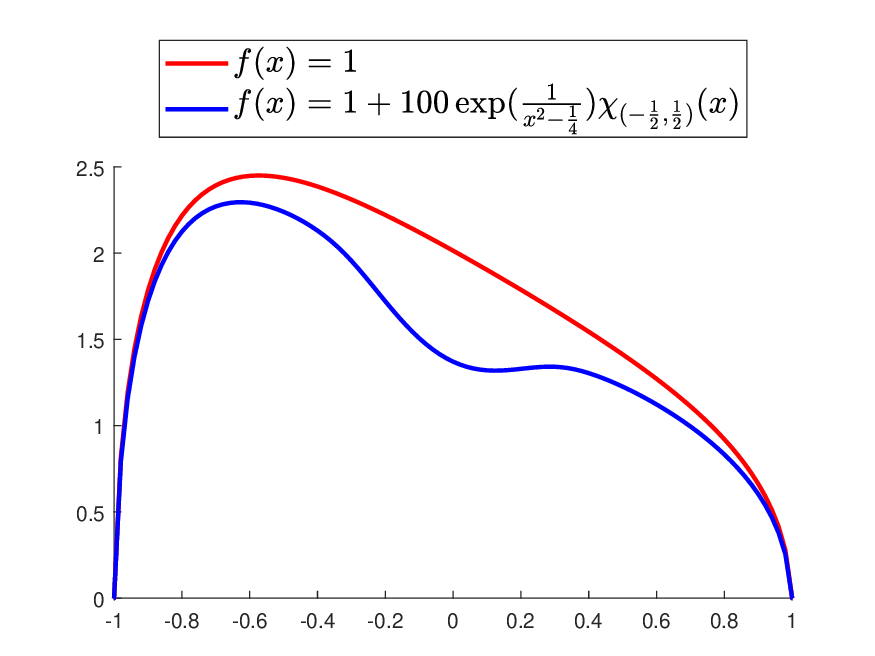}
\caption{Plot of $u_{f}|_{\Omega}$}
\label{fig:examples}
\end{figure}

Suppose we randomly select $N$ points $\{X_{i}\}_{i=1}^{N}$ in $\mD=(-3,-1)\cup(1,3)$. We then consider the log-likelihood, omitting the link function $\Phi$ from \Cref{assu:link},
\begin{equation*}
\tilde{\ell}^{N}(f) = -\frac{1}{2\sigma^{2}} \sum_{i=1}^{N} \left(Y_{i} - G(f)(X_{i})\right)^{2}.
\end{equation*}
We now consider a rather naive algorithm, similar to \cite[Algorithm~2]{FKW25RandomTruncated}, inspired by the preconditioned Crank-Nicolson (pCN) method \cite{cotter2013mcmc}, which itself is a variant of the Metropolis-Hastings (MH) algorithm \cite{metropolis1953equation,hastings1970monte,tierney1998note}.

\begin{algorithm}[H]
\caption{MH-MCMC}         
\label{alg2-naive}                          
\begin{algorithmic}[1]       
\REQUIRE an initial guess $f^{(0)}$ and a resolution parameter $J_{0}\in\mN$. 
\STATE Set ${\rm status}\,(0)={\rm accept}$. 
\STATE Creating an empty sequence $\left\{{\rm status}\,(\tau)\right\}_{\tau=1}^{\infty}$. 
\FOR{$\tau=0,1,2,\cdots$}   
\STATE Generate $f(x) = \sum_{r=-2^{J_{0}}}^{2^{J_{0}}-1} f_{r}\chi_{(0,1)} (2(2^{J_{0}}x - r))$ with $f_{r}$ randomly chosen according to the standard normal distribution. 
\STATE Propose $f^{(\tau+1)} = 1 + \sqrt{1-\beta^{2}}(f^{(\tau)}-1) + \beta f$.
\IF{${\rm status}\,(\tau)={\rm accept}$} 
\STATE $\ell_{\rm current} = \tilde{\ell}^{(N)}(f^{(\tau)})$
\ENDIF  
\IF{$\tilde{\ell}^{(N)}(f^{(\tau+1)}) > \ell_{\rm current}$}
\STATE Set ${\rm status}\,(\tau+1)={\rm accept}$.  \COMMENT{accept the proposal $f^{(\tau+1)}$} 
\ELSE 
\STATE Set $f^{(\tau+1)}=f^{(\tau)}$ and ${\rm status}\,(\tau+1)={\rm reject}$. \COMMENT{reject the proposal $f^{(\tau+1)}$}
\ENDIF 
\ENDFOR 
\RETURN $\{f^{(\tau)} \}_{\tau=1}^{\infty}$ by removing all entries $f^{(\tau)}$ that correspond to ${\rm status}\,(\tau)={\rm reject}$.  
\end{algorithmic}
\end{algorithm}

In our numerical experiment, we draw $N=150$ random samples. The noise level is set to $\sigma=0.001$, the learning rate to $\beta=0.1$, and the resolution parameter to $J_{0}=3$. A total of 5,000,000 iterations were performed, taking approximately 4,440 seconds (1 hour 14 minutes) to complete. The true potential $f$ is approximated by the burn-in sample mean 
\begin{equation*}
f_{\rm burn} = \frac{1}{\lfloor T/2\rfloor} \sum_{\tau=\lfloor T/2\rfloor+1}^{T} f^{(\tau)}. 
\end{equation*}
The reconstruction of $f$ and the evolution of the log-likelihood over the iterations are displayed in \Cref{fig:MCMC-low-precision}. 

\begin{figure}[H]
\centering
\begin{subfigure}[b]{0.49\textwidth}
\centering
\includegraphics[width=\textwidth]{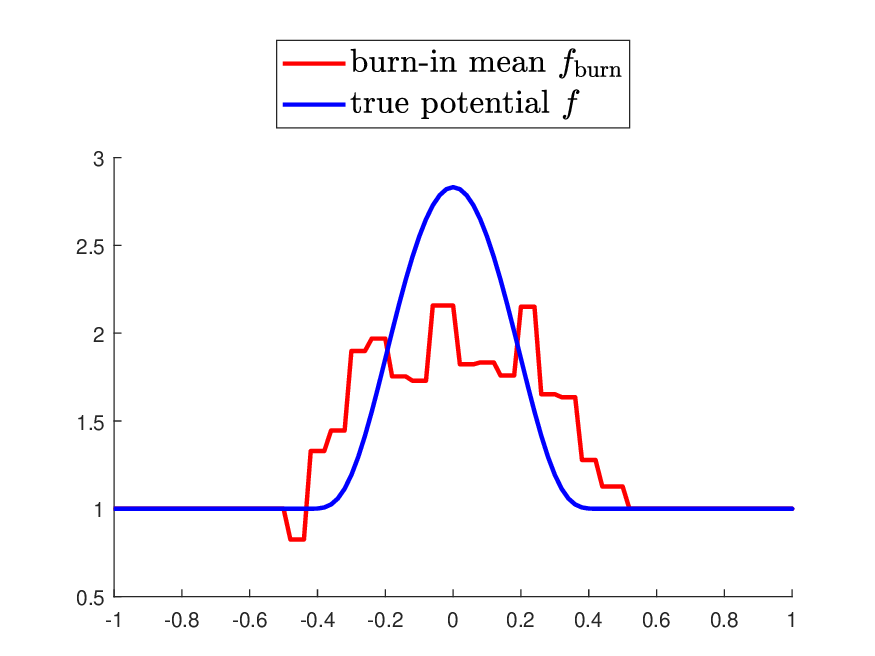}
\caption{Reconstruction of the potential $f$}
\label{fig:reconstruction}
\end{subfigure}
\hfill
\begin{subfigure}[b]{0.49\textwidth}
\centering
\includegraphics[width=\textwidth]{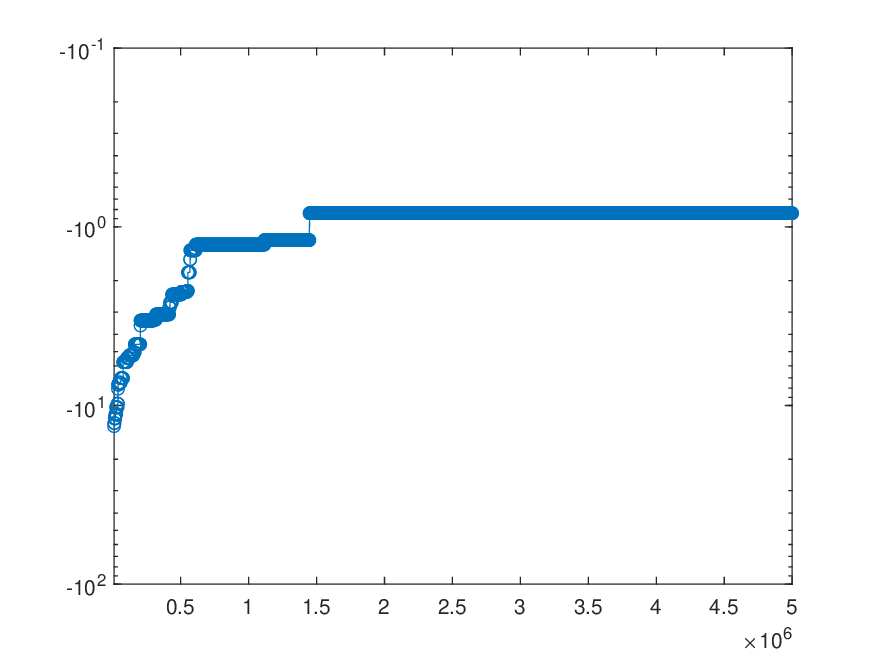}
\caption{log-likelihood $\tilde{\ell}^{N}(f)$} 
\label{fig:likelihood}
\end{subfigure}
\caption{MCMC iterations}
\label{fig:MCMC-low-precision}
\end{figure}

We must admit that the reconstruction in \Cref{fig:reconstruction} is not entirely satisfactory due to the limited precision of double-precision floating-point numbers, see \Cref{fig:likelihood}. One way to address this limitation is to use multiple-precision arithmetic (see, for example, \cite{FITI2007MultiPrecision}), though this approach significantly slows down the algorithm. Nevertheless, we recommend that readers consider the steepest-descent-type algorithms in \cite{Li24NumericalFractionalCalderon2, Li24NumericalFractionalCalderon1} (finite difference method) and \cite{dwivedi2026galerkinfiniteelementmethod} (finite element method) for improved reconstructions, as the MCMC algorithm presented here is intended primarily to illustrate the Bayesian procedure.

\subsection*{Matlab code and simulated data} 
\addtocontents{toc}{\SkipTocEntry} 

We have made our \textsc{Matlab} code package and simulated data available in the GitHub repository\footnote{\url{https://github.com/puzhaokow1993/MCMC_fractional_Laplacian}}.

\appendix 
\crefalias{section}{appendix} 

\section{Technical proofs}\label{sec_appendix}

In this Appendix, we provide proofs for the results in \Cref{sec_estimates}. The proof of \Cref{lem:forward-estimate} relies on utilizing the Alessandrini identity cleverly together with \Cref{lem:qualitative}. The key idea is to choose three different functions that appear in the Alessandrini identity such that the forward estimate can be concluded.

\begin{proof}[Proof of \Cref{lem:forward-estimate}]
Let $\phi\in C^{\infty}_c(\overline{\mD})$, $u_k=u_{f_k}$ be the unique solutions to
\begin{equation*}
    ((-\Delta)^s + f_k)u_k=0\text{ in } \Om,\quad u_k=\phi\text{ in } \Om_e
\end{equation*}
and $\tilde u_2$ solve
\begin{equation*}
    ((-\Delta)^s + f_2)\tilde u_2=0\text{ in } \Om,\quad \tilde u_2=\psi\text{ in } \Om_e.
\end{equation*}
Let us choose $\psi = \varphi(-\Delta)^s(u_1 - u_2)$, where $\varphi$ is a cutoff such that $0\leq\varphi\leq 1$, $\varphi\equiv1$ in $\mD$ and $\supp\varphi\subset\Om_e$. Next, we use the weak definition of the DN map and recall the Alessandrini identity (see, for example, \cite{GSU20Calderon})
\begin{equation}
\langle(\Lambda_{f_1}-\Lambda_{f_2})\phi,\psi\rangle = ((f_1-f_2)u_1,\tilde u_2), 
\end{equation}
where $\langle\cdot,\cdot\rangle$ denotes the distributional pairing, and $(\cdot,\cdot)$ denotes the $L^{2}$ inner product. Also, we continue the functions $f_1,f_2$ by zero outside of $\Om$. 
Now by the above \Cref{lem:qualitative}, the DN map is smooth where $\phi$ is smooth and, thus,
\begin{equation*}
\langle \Lambda_{f_k}\phi,\psi\rangle =((-\Delta)^su_k,\psi) = \int_{\Omega_e}(-\Delta)^su_k\psi\,\rmd x.
\end{equation*}
Notice that a priori, the DN map acts as a distribution, but the above identity can be made precise by an approximation argument. 
Therefore,
\begin{align*}
&\abs{\langle(\Lambda_{f_1}-\Lambda_{f_2})\phi,\psi\rangle} = \abs{\langle\Lambda_{f_1}\phi,\psi\rangle - \langle\Lambda_{f_2}\phi,\psi\rangle} =\left| \int_{\Omega_e}\psi(-\Delta)^s(u_1-u_2)\,\rmd x \right|\\
    &\quad = \left| \int_{\mD}\left((-\Delta)^s(u_1-u_2)\right)^2\,\rmd x + \int_{\supp(\varphi)\setminus\mD}\varphi\left((-\Delta)^s(u_1-u_2)\right)^2\,\rmd x \right| \\
    &\quad \ge \left| \int_{\mD}\left((-\Delta)^s(u_1-u_2)\right)^2\,\rmd x\right| = \norm{(-\Delta)^s(u_1-u_2)}_{L^2(\mD)}^2.
\end{align*}

We obtain, by utilizing the Alessandrini identity and denoting $f=f_1-f_2$, that 
\begin{align*}
&\abs{\langle(\Lambda_{f_1}-\Lambda_{f_2})\phi,\psi\rangle} = (fu_1,\tilde u_2) \\
&\quad \le \norm{f}_{L^2(\Omega)}\norm{u_1\tilde u_2}_{L^2(\Omega)} \le \norm{f}_{L^2(\Omega)}\norm{u_1}_{L^{\infty}(\Omega)}\norm{\tilde u_2}_{L^2(\Omega)},
\end{align*}
where we used the Cauchy-Schwarz inequality in the first estimate and Hölder's inequality in the second estimate. 
Let us then focus on the term $\norm{\tilde u_2}_{L^2(\Omega)}$. Notice that $\tilde u_2-\psi$ is supported in $\Om$, $(\tilde u_2-\psi)|_{\Omega}=\tilde u_2|_{\Omega}$ and, thus,
\begin{align*}
    \norm{\tilde u_2 - \psi}_{H^s(\mR^d)}\geq \norm{\tilde u_2 - \psi}_{L^2(\Omega)} = \norm{\tilde u_2}_{L^2(\Omega)}.
\end{align*}
Using \cite[Lemma 5.3]{covi2022globalinversefractionalconductivity}, we have that 
\begin{align*}
    \norm{\tilde u_2 - \psi}_{H^s(\mR^d)}\leq C(d,s,\abs{\supp(\varphi)},\Om,\tau)\norm{\psi}_{L^2(\supp(\psi))}
\end{align*}
where $\tau := \dist(\Omega,\supp(\psi))$.
Hence, by using H\"{o}lder's inequality, we obtain 
\begin{align*}
&\norm{\tilde u_2}_{L^2(\Omega)} \le C(d,s,\abs{\supp(\varphi)},\Om,\tau)\norm{\psi}_{L^2(\supp(\psi))} \\
&\quad \le C(d,s,\abs{\supp(\varphi)},\Om,\tau)\norm{\varphi}_{L^{\infty}(\supp(\varphi))}\norm{(-\Delta)^s(u_1-u_2)}_{L^2(\supp(\varphi))}.
\end{align*}
Putting the above together, we have
\begin{equation} 
\begin{aligned}
    &\norm{(-\Delta)^s(u_1-u_2)}_{L^2(\mD)}^2\\
    &\leq C(d,s,\abs{\supp(\varphi)},\Om,\tau)\norm{f}_{L^2(\Om)}\norm{u_1}_{L^{\infty}(\Om)}\norm{\varphi}_{L^{\infty}(\supp(\varphi))}\norm{(-\Delta)^s(u_1-u_2)}_{L^2(\supp(\varphi))}\\
    &\leq C(d,s,\abs{\supp(\varphi)},\Om,\tau,\phi)\norm{f}_{L^2(\Om)}\norm{\varphi}_{L^{\infty}(\supp(\varphi))}\norm{(-\Delta)^s(u_1-u_2)}_{L^2(\supp(\varphi))}
\end{aligned} \label{u_1_u_2_difference_1}
\end{equation}
where we used the estimate \eqref{eq:L-infty-bound}.

The estimate \eqref{u_1_u_2_difference_1} holds for any smooth cutoff $\varphi$. Let us define two functions $\Phi(x)=\chi_{B(0,1)}e^{\frac{1}{\abs{x}^2-1}}$ and $\Phi_1(x)=\frac{\Phi(x)}{\int_{\mR^d}\Phi(x)\,\rmd x}$. Let now $\e>0$ and set $\Phi_{\e}=\frac{1}{\e^d}\Phi_1\left(\frac{x}{\e}\right)$, $\varphi_{\e}=\Phi_{\e}\ast \chi_{\mD}$. Then $\norm{\varphi_{\e}}_{L^{\infty}(\supp(\varphi_{\e})}\leq 1$, $\varphi_{\e}$ is smooth,
\begin{equation*}
    \supp(\varphi_{\e})\subseteq\mD_{\e}=\{x: \dist(x,\mD)\leq\e\}.
\end{equation*}
and, by making $\e$ smaller if necessary, $\supp(\varphi_{\e})\cap\overline{\Om}=\emptyset$.
Moreover, $\lim_{\e\to0}\varphi_{\e} = \chi_{\mD}$. For all small enough $\e > 0$, the estimate \eqref{u_1_u_2_difference_1} holds with $\varphi$ replaced by $\varphi_{\e}$. Thus, taking the limit $\e\to0$, we have
\begin{equation*}
    \norm{(-\Delta)^s(u_1-u_2)}_{L^2(\mD)}^2\leq C(d,s,\abs{\mD},\Om,\dist(\Omega,\mD),\phi)\norm{f}_{L^2(\Om)}\norm{(-\Delta)^s(u_1-u_2)}_{L^2(\mD)}.
\end{equation*}
This implies, after recalling that $f=f_1-f_2$, 
\begin{equation}\label{lemma_estimate}
    \norm{(-\Delta)^s(u_1-u_2)}_{L^2(\mD)}\leq C(d,s,\abs{\mD},\Om,\dist(\Omega,\mD),\phi)\norm{f_1-f_2}_{L^2(\Om)}.
\end{equation}
This readily gives the claimed estimate \eqref{forward_estimate}.
\end{proof}

We then prove the $L^\infty$ boundedness of the forward problem. The proof uses the singular integral definition of the fractional Laplacian. 

\begin{proof}[Proof of \Cref{lem:uniform-bound}]
Since $\phi=0$ in $\Omega$, for each $x\in\mD$, we see that 
\begin{equation*}
\begin{aligned} 
& (-\Delta)^{s}u_{f}(x) = C_{d,s}\int_{\Omega}\frac{u_{f}(x)-u_{f}(y)}{\abs{x-y}^{d+2s}}\,\rmd y + C_{d,s}\int_{\Omega_{e}}\frac{u_{f}(x)-u_{f}(y)}{\abs{x-y}^{d+2s}}\,\rmd y \\ 
& \quad = C_{d,s}\int_{\Omega}\frac{\phi(x)-u_{f}(y)}{\abs{x-y}^{d+2s}}\,\rmd y + C_{d,s}\int_{\Omega_{e}}\frac{\phi(x)-\phi(y)}{\abs{x-y}^{d+2s}}\,\rmd y \\ 
& \quad = C_{d,s}\int_{\Omega}\frac{\phi(x)-u_{f}(y)}{\abs{x-y}^{d+2s}}\,\rmd y + C_{d,s}\int_{\mR^{d}}\frac{\phi(x)-\phi(y)}{\abs{x-y}^{d+2s}}\,\rmd y - C_{d,s}\int_{\Om}\frac{\phi(x)-\phi(y)}{\abs{x-y}^{d+2s}}\,\rmd y \\ 
& \quad = C_{d,s}\int_{\Omega}\frac{\phi(x)-u_{f}(y)}{\abs{x-y}^{d+2s}}\,\rmd y + (-\Delta)^{s}\phi(x) - C_{d,s}\phi(x)\int_{\Om}\frac{1}{\abs{x-y}^{d+2s}}\,\rmd y, 
\end{aligned} 
\end{equation*}
where 
\begin{equation*}
C_{d,s} := \frac{4^{s}\Gamma(d/2+s)}{\pi^{d/2}\abs{\Gamma(-s)}}. 
\end{equation*}
Since for each $x\in\mD$, one can estimate 
\begin{equation*}
\begin{aligned}
& \left| \int_{\Omega}\frac{\phi(x)-u_{f}(y)}{\abs{x-y}^{d+2s}}\,\rmd y \right| \le \frac{1}{\dist(\Omega,\mD)^{d+2s}} \int_{\Omega} \abs{\phi(x)-u_{f}(y)}\,\rmd y \\ 
& \quad \le \frac{\abs{\Omega}}{\dist(\Omega,\mD)^{d+2s}} (\phi(x) + \norm{u_{f}}_{L^{\infty}(\Omega)}) \\ 
& \quad \le \frac{2\abs{\Omega}}{\dist(\Omega,\mD)^{d+2s}} \norm{\phi}_{L^{\infty}(\Omega_{e})},
\end{aligned}
\end{equation*}
where we used \eqref{eq:L-infty-bound} in the last inequality. Similarly for the last term, we have 
\begin{equation*}
    \left|\phi(x)\int_{\Om}\frac{1}{\abs{x-y}^{d+2s}}\,\rmd y\right|\leq \frac{\abs{\Om}}{\dist(\Omega,\mD)^{d+2s}}\norm{\phi}_{L^{\infty}(\Om_e)}.
\end{equation*}
Therefore, when we combine all the equations above, we obtain 
\begin{equation*}
\norm{(-\Delta)^{s}u_{f}}_{L^{\infty}(\mD)} \le C(d,s,\abs{\Omega},\dist(\Omega,\mD))\left( \norm{\phi}_{L^{\infty}(\Omega_{e})} + \norm{(-\Delta)^{s}\phi}_{L^{\infty}(\mD)} \right), 
\end{equation*}
which concludes our lemma. 
\end{proof}

Before proving \Cref{prop:inverse-estimate}, we first state the following lemma, which is a special case of \cite[Proposition~3.3]{Rueland2021SingleMeasurement}. This can be proved using \cite[Proposition~6.1]{GRSU20Reconstruction}.

\begin{lemma}\label{lem:inverse-estimate1}
Let $d\in\mN$ and $0<s<1$. Let $\Omega$ and $\mD$ be nonempty bounded smooth domains in $\mR^{d}$ with $\overline{\mD}\cap\overline{\Omega}=\emptyset$. Suppose that $\phi\in C^{\infty}_c(\Omega_{e})$, $M_{0}>0$ and $f_1, f_2 \in L^{\infty}(\Om)$ are such that $f_1, f_2 \geq0$, and 
\begin{equation*}
\norm{f_{k}}_{L^{\infty}(\Omega)} \le M_{0} \quad \text{for all $k=1,2$.}
\end{equation*}
Then there exist constants $\tilde{C}_{\rm stab},\tilde{E},\tilde{\mu}>0$ (depending only on $\Omega,\mD,d,s,M_{0},\phi$) such that 
\begin{equation*}
\begin{aligned} 
&\norm{(-\Delta)^{s}(u_{f_{1}}-u_{f_{2}})}_{H^{-s}(\Omega)}+\norm{u_{f_{1}}-u_{f_{2}}}_{H_{\overline{\Omega}}^{s}} \\
& \quad \le \tilde{C}_{\rm stab}^{2}\left|\log\left(\tilde{E}\norm{G(f_{1})-G(f_{2})}_{H^{-s}(\mD)}^{-1}\right)\right|^{-2\tilde{\mu}}. 
\end{aligned} 
\end{equation*}
\end{lemma} 

Let us recall \cite[Lemma~5.1]{Rueland2021SingleMeasurement}, which is a slight variant of the weighted boundary interpolation result from \cite{ASV13FarFieldStable}. 

\begin{lemma}\label{lem:interpolation}
Let $d\in\mN$. Let $\Omega$ and $\mO$ be nonempty bounded smooth domains in $\mR^{d}$ with $\overline{\mO}\subset\Omega$. Let $r_{0}:=\frac{1}{2}\dist(\mO,\Omega_{e})$. Let $w\in L^{2}(\Omega)$ be such that for each $x_{0}\in\mO$ and $r\in(0,r_{0})$, we have 
\begin{equation}
\norm{w}_{L^{2}(B_{r}(x_{0}))} \ge C_{\rm low}r^{b} \label{eq:polybound-u}
\end{equation}
for some constants $C_{\rm low}>0$ and $b>0$. Let $f\in C^{s}(\overline{\Omega})$ for some $0<s<1$. Assume that there exist constants $C_{\rm stab}>0$, $\tilde{\mu}>0$ such that  
\begin{equation}
\norm{fw}_{L^{2}(\mO)} \le C_{\rm stab}E \left| \log(E\epsilon^{-1}) \right|^{-\tilde{\mu}}, \label{eq:interpolation1}
\end{equation}
for some $E>0$. Then there exists a constant $\tilde{C}=\tilde{C}(\Omega,\mO,C_{\rm low},C_{\rm stab},d,b,s)>0$ such that  
\begin{equation*}
\norm{f}_{L^{\infty}(\mO)} \le \tilde{C}E^{\frac{s}{s+b}} \norm{f}_{C^{s}(\Omega)}^{\frac{s}{s+b}} \left| \log(E\epsilon^{-1}) \right|^{-\frac{s}{s+b}\tilde{\mu}}
\end{equation*}
for all $\epsilon\in(0,1/2)$. 
\end{lemma}

From \cite[Remark~1 in page~143]{Tri83FunctionSpace}, we also have the inequalities  
\begin{equation}
\norm{fg}_{H_{0}^{\alpha}(\Omega)} \le C \norm{f}_{C^{\alpha}(\Om)}\norm{g}_{H_{0}^{\alpha}(\Omega)} ,\quad \norm{fg}_{C^{\alpha}(\Om)} \le C \norm{f}_{C^{\alpha}(\Om)}\norm{g}_{C^{\alpha}(\Om)} \quad \text{for any $\alpha\ge 0$}, \label{eq:algebra}
\end{equation}
which can also be found in \cite[(11)]{nickl2020bernstein}. As an immediate consequence, we see that 
\begin{equation}
\begin{aligned} 
& \norm{fg}_{H^{-\alpha}(\Om)} = \sup_{\norm{\phi}_{H_{0}^{\alpha}(\Omega)}=1} \int fg\phi \le \norm{f}_{H^{-\alpha}(\Om)} \sup_{\norm{\phi}_{H_{0}^{\alpha}(\Omega)}=1} \norm{g\phi}_{H_{0}^{\alpha}(\Omega)} \\ 
& \quad \overset{\eqref{eq:algebra}}{\le} \norm{f}_{H^{-\alpha}(\Om)} \sup_{\norm{\phi}_{H_{0}^{\alpha}(\Omega)}=1} \norm{g}_{C^{\alpha}(\Om)} \norm{\phi}_{H_{0}^{\alpha}(\Omega)} = \norm{f}_{H^{-\alpha}(\Om)}\norm{g}_{C^{\alpha}(\Om)} \quad \text{for any $\alpha\ge 0$.}
\end{aligned} \label{eq:algebra-dual}  
\end{equation}
We are now ready to prove \Cref{prop:inverse-estimate}. 

\begin{proof}[Proof of \Cref{prop:inverse-estimate}] 
It was proved in \cite[(32)]{Rueland2021SingleMeasurement} that for $w=u_{f_{1}}$, \eqref{eq:polybound-u} holds true for some constants $C_{\rm low}>0$ and $b>0$, depending only on $\Omega,\mO,d,M_{0},\phi$. We now deduce a bound as in \eqref{eq:interpolation1} for $w=u_{f_{1}}$ and $f=f_{1}-f_{2}$. We first estimate 
\begin{equation*}
\norm{u_{f_{1}}(f_{1}-f_{2})}_{L^{2}(\Omega)} \le C(\Omega,s)\norm{u_{f_{1}}(f_{1}-f_{2})}_{H_{0}^{s}(\Omega)}^{\frac{1}{2}}\norm{u_{f_{1}}(f_{1}-f_{2})}_{H^{-s}(\Omega)}^{\frac{1}{2}}
\end{equation*}
as in \cite[(30)]{Rueland2021SingleMeasurement}. The first term on the right-hand side can be controlled using \eqref{eq:algebra}, and one obtains  
\begin{equation*}
\begin{aligned}
& \norm{u_{f_{1}}(f_{1}-f_{2})}_{H_{0}^{s}(\Omega)} \le C(\Omega,s) \norm{u_{f_{1}}}_{H^{s}(\mR^{d})}\left( \norm{f_{1}}_{C^{s}(\Omega)} \vee \norm{f_{2}}_{C^{s}(\Omega)} \right) \\ 
& \quad \le C(\Omega,s,M_{0},\phi) \left( \norm{f_{1}}_{C^{s}(\Omega)} \vee \norm{f_{2}}_{C^{s}(\Omega)} \right).
\end{aligned}
\end{equation*}
On the other hand, the second term on the right-hand side can be estimated using \eqref{eq:algebra-dual} and \Cref{lem:inverse-estimate1}. This gives 
\begin{equation*}
\begin{aligned}
& \norm{u_{f_{1}}(f_{1}-f_{2})}_{H^{-s}(\Omega)} = \norm{u_{f_{1}}f_{1}-u_{f_{2}}f_{2} - f_{2}(u_{f_{1}}-u_{f_{2}})}_{H^{-s}(\Omega)} \\ 
& \quad \le \norm{(-\Delta)^{s}(u_{f_{1}}-u_{f_{2}})}_{H^{-s}(\Omega)} + \norm{f_{2}}_{C^{s}(\Omega)}\norm{u_{f_{1}}-u_{f_{2}}}_{H^{-s}(\Omega)} \\ 
& \quad \le \tilde{C}_{\rm stab}^{2}\tilde{E}^{2}(1+\norm{f_{2}}_{C^{s}(\Omega)})\left|\log\left(\tilde{E}\norm{G(f_{1})-G(f_{2})}_{H^{-s}(\mD)}^{-1}\right)\right|^{-2\tilde{\mu}},
\end{aligned}
\end{equation*}
for some constants $\tilde{C}_{\rm stab},\tilde{E},\tilde{\mu}>0$ (depending only on $\Omega,\mD,d,s,M_{0},\phi$). Combining the above three equations, we obtain 
\begin{equation}
\begin{aligned} 
& \norm{u_{f_{1}}(f_{1}-f_{2})}_{L^{2}(\Omega)} \\
& \quad \le C_{\rm stab}(1+\norm{f_{1}}_{C^{s}(\Omega)} \vee \norm{f_{2}}_{C^{s}(\Omega)}) \left|\log\left(\tilde{E}\norm{G(f_{1})-G(f_{2})}_{H^{-s}(\mD)}^{-1}\right)\right|^{-\tilde{\mu}}
\end{aligned} 
\end{equation}
for some $C_{\rm stab},\tilde{E},\tilde{\mu}>0$ (depending only on $\Omega,\mD,d,s,M_{0},\phi$), which verifies \eqref{eq:interpolation1} with 
\begin{equation*}
\begin{aligned} 
E &= 1+\norm{f_{1}}_{C^{s}(\Omega)} \vee \norm{f_{2}}_{C^{s}(\Omega)}, \\ 
\epsilon &= (1+\norm{f_{1}}_{C^{s}(\Omega)} \vee \norm{f_{2}}_{C^{s}(\Omega)})\tilde{E}^{-1}\norm{G(f_{1})-G(f_{2})}_{H^{-s}(\mD)}. 
\end{aligned} 
\end{equation*}
Now we use \Cref{lem:interpolation} to conclude \eqref{eq:stability-quantify1} with $\gamma=\frac{2s}{s+b}$, $\mu=\frac{s}{s+b}\tilde{\mu}>0$, provided 
\begin{equation*}
\epsilon = (1+\norm{f_{1}}_{C^{s}(\Omega)} \vee \norm{f_{2}}_{C^{s}(\Omega)})\tilde{E}^{-1}\norm{G(f_{1})-G(f_{2})}_{H^{-s}(\mD)} < \frac{1}{2}. 
\end{equation*} 
We now show that \eqref{eq:stability-quantify1} remains to be true when 
\begin{equation}
\frac{1}{2(1+\norm{f_{1}}_{C^{s}(\Omega)} \vee \norm{f_{2}}_{C^{s}(\Omega)})} \le \tilde{E}^{-1}\norm{G(f_{1})-G(f_{2})}_{H^{-s}(\mD)} < \frac{1}{2}. \label{eq:stability-case1}
\end{equation}
We see that 
\begin{equation*}
\log \left( \tilde{E}\norm{G(f_{1})-G(f_{2})}_{H^{-s}(\mD)}^{-1} \right) \ge \log 2 > 0 
\end{equation*}
and 
\begin{equation*}
\log \left( \tilde{E}\norm{G(f_{1})-G(f_{2})}_{H^{-s}(\mD)}^{-1} \right) \le \log \left(2(1+\norm{f_{1}}_{C^{s}(\Omega)} \vee \norm{f_{2}}_{C^{s}(\Omega)})\right).  
\end{equation*}
Hence, 
\begin{equation*}
\begin{aligned} 
& \left(1+\norm{f_{1}}_{C^{s}(\Omega)}^{\gamma} \vee \norm{f_{2}}_{C^{s}(\Omega)}^{\gamma}\right)\left|\log \left( \tilde{E}\norm{G(f_{1})-G(f_{2})}_{H^{-s}(\mD)}^{-1} \right) \right|^{-\mu}  \\ 
& \quad \ge \left(1+\norm{f_{1}}_{C^{s}(\Omega)}^{\gamma} \vee \norm{f_{2}}_{C^{s}(\Omega)}^{\gamma}\right)\left| \log \left(2(1+\norm{f_{1}}_{C^{s}(\Omega)} \vee \norm{f_{2}}_{C^{s}(\Omega)})\right) \right|^{-\mu} \\ 
& \quad \ge \inf_{t\ge 0} (1+t^{\gamma})|\log(2(1+t))|^{-\mu} > 0. 
\end{aligned} 
\end{equation*}
Combining the above inequality with the trivial estimate $\norm{f_{1}-f_{2}}_{L^{\infty}(\mO)} \le 2M_{0}$, we now conclude that \eqref{eq:stability-quantify1} holds true when 
\begin{equation}
\tilde{E}^{-1}\norm{G(f_{1})-G(f_{2})}_{H^{-s}(\mD)} < \frac{1}{2}. \label{eq:case1}
\end{equation}
Finally, we again apply the trivial estimate $\norm{f_{1}-f_{2}}_{L^{\infty}(\mO)} \le 2M_{0}$ to conclude \eqref{eq:stability-quantify1} in the case where \eqref{eq:case1} does not hold. 
\end{proof} 

\subsection*{Acknowledgments} 

Kow was partially supported by the National Science and Technology Council of Taiwan, NSTC 112-2115-M-004-004-MY3 and NSTC 115-2628-M-004-001, and by the National Center for Theoretical Sciences of Taiwan. Nurminen and Railo were supported by the Research Council of Finland through the Flagship of Advanced Mathematics for Sensing, Imaging and Modelling (grant numbers 359183 and 359208) and by the Emil Aaltonen Foundation. In addition, Railo was supported by the Fulbright Finland Foundation (ASLA-Fulbright Research Grant for Junior Scholars 2024-2025) at Stanford University, and the Jenny and Antti Wihuri Foundation. 

\end{sloppypar}

\bibliographystyle{custom}
\bibliography{ref}

\newcommand{\etalchar}[1]{$^{#1}$}
\begin{thebibliography}{OHKSW24}

\bibitem[AN19]{abraham2020statistical}
K.~Abraham and R.~Nickl.
\newblock On statistical {C}alder{\'o}n problems.
\newblock {\em Math. Stat. Learn.}, 2(2):165--216, 2019.
\newblock \href{https://mathscinet.ams.org/mathscinet/article?mr=4130599}{MR4130599}, \href{https://zbmath.org/1445.35144}{Zbl:1445.35144}, \href{https://doi.org/10.4171/MSL/14}{doi:10.4171/MSL/14}, \href{https://arxiv.org/abs/1906.03486}{\texttt{arXiv:1906.03486}}.

\bibitem[ASV13]{ASV13FarFieldStable}
G.~Alessandrini, E.~Sincich, and S.~Vessella.
\newblock Stable determination of surface impedance on a rough obstacle by far field data.
\newblock {\em Inverse Probl. Imaging}, 7(2):341--351, 2013.
\newblock \href{https://mathscinet.ams.org/mathscinet/article?mr=3063537}{MR3063537}, \href{https://zbmath.org/1267.35248}{Zbl:1267.35248}, \href{https://doi.org/10.3934/ipi.2013.7.341}{doi:10.3934/ipi.2013.7.341}, \href{https://arxiv.org/abs/1207.5421}{\texttt{arXiv:1207.5421}}.

\bibitem[Boh23]{Boh2023BvM}
J.~Bohr.
\newblock A bernstein-von-{M}ises theorem for the {C}alder{\'o}n problem with piecewise constant conductivities.
\newblock {\em Inverse Problems}, 39(1), 2023.
\newblock Paper No. 015002, 18 pages. \href{https://mathscinet.ams.org/mathscinet/article?mr=4527131}{MR4527131}, \href{https://zbmath.org/1529.35567}{Zbl:1529.35567}, \href{https://doi.org/10.1088/1361-6420/ac9db8}{doi:10.1088/1361-6420/ac9db8}, \href{https://arxiv.org/abs/2206.08177}{\texttt{arXiv:2206.08177}}.

\bibitem[BV16]{bucur_valdinoci}
C.~Bucur and E.~Valdinoci.
\newblock {\em Nonlocal diffusion and applications}, volume~20 of {\em Lect. Notes Unione Mat. Ital.}
\newblock Springer/Unione Matematica Italiana, Cham/Bologna, 2016.
\newblock \href{https://mathscinet.ams.org/mathscinet/article?mr=3469920}{MR3469920}, \href{https://zbmath.org/1377.35002}{Zbl:1377.35002}, \href{https://doi.org/10.1007/978-3-319-28739-3}{doi:10.1007/978-3-319-28739-3}, \href{https://arxiv.org/abs/1504.08292}{\texttt{arXiv:1504.08292}}.

\bibitem[CRSW13]{cotter2013mcmc}
S.~L. Cotter, G.~O. Roberts, A.~M. Stuart, and D.~White.
\newblock {M}{C}{M}{C} methods for functions: modifying old algorithms to make them faster.
\newblock {\em Statist. Sci.}, 28(3):424--446, 2013.
\newblock \href{https://mathscinet.ams.org/mathscinet/article?mr=3135540}{MR3135540}, \href{https://zbmath.org/1331.62132}{Zbl:1331.62132}, \href{https://doi.org/10.1214/13-STS421}{doi:10.1214/13-STS421}, \href{https://arxiv.org/abs/1202.0709}{\texttt{arXiv:1202.0709}}.

\bibitem[CMR21]{CMR20unique}
G.~Covi, K.~M{\"o}nkk{\"o}nen, and J.~Railo.
\newblock Unique continuation property and {P}oincar\'e inequality for higher order fractional {L}aplacians with applications in inverse problems.
\newblock {\em Inverse Probl. Imaging}, 14(4):641--681, 2021.
\newblock \href{https://mathscinet.ams.org/mathscinet/article?mr=4259671}{MR4259671}, \href{https://zbmath.org/1471.35327}{Zbl:1471.35327}, \href{http://dx.doi.org/10.3934/ipi.2021009}{doi:10.3934/ipi.2021009}, \href{https://arxiv.org/abs/2001.06210}{\texttt{arXiv:2001.06210}}.

\bibitem[CMRU22]{CMRU2022higherorder}
G.~Covi, K.~M{\"o}nkk{\"o}nen, J.~Railo, and G.~Uhlmann.
\newblock The higher order fractional {C}alder{\'o}n problem for linear local operators: uniqueness.
\newblock {\em Adv. Math.}, 399, 2022.
\newblock Paper No. 108246, 29 pages. \href{https://mathscinet.ams.org/mathscinet/article?mr=4383014}{MR4383014}, \href{https://zbmath.org/1486.35462}{Zbl:1486.35462}, \href{https://doi.org/10.1016/j.aim.2022.108246}{doi:10.1016/j.aim.2022.108246}, \href{https://arxiv.org/abs/2008.10227}{\texttt{arXiv:2008.10227}}.

\bibitem[CRZ26]{covi2022globalinversefractionalconductivity}
G.~Covi, J.~Railo, and P.~Zimmermann.
\newblock The global inverse fractional conductivity problem.
\newblock {\em arXiv preprint}, 2026.
\newblock To appear in Calc. Var. Partial Differential Equations. \href{https://arxiv.org/abs/2204.04325}{\texttt{arXiv:2204.04325}}.

\bibitem[DvWZ18]{DvWZ18FDMFractionalLaplacian}
S.~Duo, H.~W. van Wyk, and Y.~Zhang.
\newblock A novel and accurate finite difference method for the fractional {L}aplacian and the fractional {P}oisson problem.
\newblock {\em J. Comput. Phys.}, 355:233--252, 2018.
\newblock \href{https://mathscinet.ams.org/mathscinet/article?mr=3738575}{MR3738575}, \href{https://zbmath.org/1380.65323}{Zbl:1380.65323}, \href{https://doi.org/10.1016/j.jcp.2017.11.011}{doi:10.1016/j.jcp.2017.11.011}.

\bibitem[DZ19]{DZ19FDMFractionalLaplacian}
S.~Duo and Y.~Zhang.
\newblock Accurate numerical methods for two and three dimensional integral fractional {L}aplacian with applications.
\newblock {\em Comput. Methods Appl. Mech. Engrg.}, 355:639--662, 2019.
\newblock \href{https://mathscinet.ams.org/mathscinet/article?mr=3979209}{MR3979209}, \href{https://zbmath.org/1441.65085}{Zbl:1441.65085}, \href{https://doi.org/10.1016/j.cma.2019.06.016}{doi:10.1016/j.cma.2019.06.016}, \href{https://arxiv.org/abs/1804.02718}{\texttt{arXiv:1804.02718}}.

\bibitem[DRR26]{dwivedi2026galerkinfiniteelementmethod}
M.~Dwivedi, J.~Railo, and A.~Rupp.
\newblock A {G}alerkin finite element method for the fractional {C}alder\'on problem.
\newblock {\em arXiv preprint}, 2026.
\newblock \href{https://doi.org/10.48550/arXiv.2603.26287}{\texttt{arXiv:2603.26287}}.

\bibitem[FGK{\etalchar{+}}25]{feizmohammadi2025fractional}
A.~Feizmohammadi, T.~Ghosh, K.~Krupchyk, A.~R{\"u}land, J.~Sj{\"o}strand, and G.~Uhlmann.
\newblock Fractional anisotropic {C}alder{\'o}n problem with external data.
\newblock {\em arXiv preprint}, 2025.
\newblock \href{https://arxiv.org/abs/2502.00710}{\texttt{arXiv:2502.00710}}.

\bibitem[FGKU25]{FGKU2025fractionalanisotropic}
A.~Feizmohammadi, T.~Ghosh, K.~Krupchyk, and G.~Uhlmann.
\newblock Fractional anisotropic {C}alder{\'o}n problem on closed {R}iemannian manifolds.
\newblock {\em J. Differential Geom.}, 131(2):401--414, 2025.
\newblock \href{https://mathscinet.ams.org/mathscinet/article?mr=4955611}{MR4955611}, \href{https://zbmath.org/8097552}{Zbl:8097552}, \href{https://doi.org/10.4310/jdg/1757353909}{doi:10.4310/jdg/1757353909}, \href{https://arxiv.org/abs/2112.03480}{\texttt{arXiv:2112.03480}}.

\bibitem[FITI07]{FITI2007MultiPrecision}
H.~Fujiwara, H.~Imai, T.~Takeuchi, and Y.~Iso.
\newblock Numerical treatment of analytic continuation with multiple-precision arithmetic.
\newblock {\em Hokkaido Math. J.}, 36(4):837--847, 2007.
\newblock \href{https://mathscinet.ams.org/mathscinet/article?mr=2378294}{MR2378294}, \href{https://zbmath.org/1142.65083}{Zbl:1142.65083}, \href{https://doi.org/10.14492/hokmj/1272848036}{doi:10.14492/hokmj/1272848036}.

\bibitem[FKW24]{FKW24ProbabilityInverseScattering}
T.~Furuya, P.-Z. Kow, and J.-N. Wang.
\newblock Consistency of the {B}ayes method for the inverse scattering problem.
\newblock {\em Inverse Problems}, 40(5), 2024.
\newblock Paper No. 055001, 24 pages. \href{https://mathscinet.ams.org/mathscinet/article?mr=4723841}{MR4723841}, \href{https://zbmath.org/7867314}{Zbl:7867314}, \href{https://doi.org/10.1088/1361-6420/ad3089}{doi:10.1088/1361-6420/ad3089}.

\bibitem[FKW26]{FKW25RandomTruncated}
T.~Furuya, P.-Z. Kow, and J.-N. Wang.
\newblock Consistency of the {B}ayes method for the inverse scattering problem with randomly truncated sieve priors.
\newblock {\em Inverse Probl. Imaging.}, 21:222--244, 2026.
\newblock \href{https://mathscinet.ams.org/mathscinet/article?mr=4997981}{MR4997981}, \href{https://zbmath.org/1583.35410}{Zbl:1583.35410}, \href{https://doi.org/10.3934/ipi.2025040}{doi:10.3934/ipi.2025040}.

\bibitem[GRSU20]{GRSU20Reconstruction}
T.~Ghosh, A.~R{\"u}land, M.~Salo, and G.~Uhlmann.
\newblock Uniqueness and reconstruction for the fractional {C}alder{\'o}n problem with a single measurement.
\newblock {\em J. Funct. Anal.}, 279(1), 2020.
\newblock 108505, 42 pages. \href{https://mathscinet.ams.org/mathscinet/article?mr=4083776}{MR4083776}, \href{https://zbmath.org/1452.35255}{Zbl:1452.35255}, \href{https://doi.org/10.1016/j.jfa.2020.108505}{doi:10.1016/j.jfa.2020.108505}, \href{https://arxiv.org/abs/1801.04449}{\texttt{arXiv:1801.04449}}.

\bibitem[GSU20]{GSU20Calderon}
T.~Ghosh, M.~Salo, and G.~Uhlmann.
\newblock The {C}alder{\'o}n problem for the fractional {S}chr{\"o}dinger equation.
\newblock {\em Anal. PDE}, 13(2):455--475, 2020.
\newblock \href{https://mathscinet.ams.org/mathscinet/article?mr=4078233}{MR4078233}, \href{https://zbmath.org/1439.35530}{Zbl:1439.35530}, \href{https://doi.org/10.2140/apde.2020.13.455}{doi:10.2140/apde.2020.13.455}, \href{https://arxiv.org/abs/1609.09248}{\texttt{arXiv:1609.09248}}.

\bibitem[GN21]{gine2021mathematical}
E.~Gin{\'e} and R.~Nickl.
\newblock {\em Mathematical foundations of infinite-dimensional statistical models}, volume~40 of {\em Camb. Ser. Stat. Probab. Math.}
\newblock Cambridge university press, New York, 2021.
\newblock \href{https://mathscinet.ams.org/mathscinet/article?mr=3588285}{MR3588285}, \href{https://zbmath.org/1460.62007}{Zbl:1460.62007}, \href{https://doi.org/10.1017/9781009022811}{doi:10.1017/9781009022811}.

\bibitem[GK20]{GK20BVM}
M.~Giordano and H.~Kekkonen.
\newblock Bernstein--von {M}ises theorems and uncertainty quantification for linear inverse problems.
\newblock {\em SIAM/ASA J. Uncertain. Quantif.}, 8(1):342--373, 2020.
\newblock \href{https://mathscinet.ams.org/mathscinet/article?mr=4069334}{MR4069334}, \href{https://zbmath.org/1436.62161}{Zbl:1436.62161}, \href{https://doi.org/10.1137/18M1226269}{doi:10.1137/18M1226269}, \href{https://arxiv.org/abs/1811.04058}{\texttt{arXiv:1811.04058}}.

\bibitem[GN20]{giordano2020consistency}
M.~Giordano and R.~Nickl.
\newblock Consistency of bayesian inference with gaussian process priors in an elliptic inverse problem.
\newblock {\em Inverse problems}, 36(8), 2020.
\newblock Paper No. 085001, 35 pages. \href{https://mathscinet.ams.org/mathscinet/article?mr=4151406}{MR4151406}, \href{https://zbmath.org/1445.35330}{Zbl:1445.35330}, \href{https://doi.org/10.1088/1361-6420/ab7d2a}{doi:10.1088/1361-6420/ab7d2a}, \href{https://arxiv.org/abs/1910.07343}{\texttt{arXiv:1910.07343}}.

\bibitem[Gru15]{Gru15FractionalLaplacianDomains}
G.~Grubb.
\newblock Fractional {L}aplacians on domains, a development of {H}{\"{o}}rmander's theory of $\mu$-transmission pseudodifferential operators.
\newblock {\em Adv. Math.}, 268:478--528, 2015.
\newblock \href{https://mathscinet.ams.org/mathscinet-getitem?mr=3276603}{MR3276603}, \href{https://zbmath.org/1318.47064}{Zbl:1318.47064}, \href{https://doi.org/10.1016/j.aim.2014.09.018}{doi:10.1016/j.aim.2014.09.018}, \href{https://arxiv.org/abs/1310.0951}{\texttt{arXiv:1310.0951}}.

\bibitem[Has70]{hastings1970monte}
W.~K. Hastings.
\newblock Monte {C}arlo sampling methods using {M}arkov chains and their applications.
\newblock {\em Biometrika}, 57(1):97--109, 1970.
\newblock \href{https://mathscinet.ams.org/mathscinet/article?mr=3363437}{MR3363437}, \href{https://zbmath.org/0219.65008}{Zbl:0219.65008}, \href{https://doi.org/10.1093/biomet/57.1.97}{doi:10.1093/biomet/57.1.97}.

\bibitem[HQD{\etalchar{+}}10]{Humphries2010_un}
N.~E. Humphries, N.~Queiroz, J.~R.~M. Dyer, N.~G. Pade, M.~K. Musyl, K.~M. Schaefer, D.~W. Fuller, J.~M. Brunnschweiler, T.~K. Doyle, J.~D.~R. Houghton, G.~C. Hays, C.~S. Jones, L.~R. Noble, V.~J. Wearmouth, Southall~E. J., and Sims~D. W.
\newblock Environmental context explains {L}{\'e}vy and {B}rownian movement patterns of marine predators.
\newblock {\em Nature}, 465(7301):1066--1069, 2010.
\newblock \href{https://doi.org/10.1038/nature09116}{doi:10.1038/nature09116}.

\bibitem[IKS25]{IKS25UCPMRT}
J.~Ilmavirta, P.-Z. Kow, and S.~K. Sahoo.
\newblock Unique continuation for the momentum ray transform.
\newblock {\em J. Fourier Anal. Appl.}, 31(2), 2025.
\newblock Paper No. 17, 30 pages. \href{https://mathscinet.ams.org/mathscinet/article?mr=4875757}{MR4875757}, \href{https://zbmath.org/8051195}{Zbl:8051195}, \href{https://doi.org/10.1007/s00041-025-10149-8}{doi:10.1007/s00041-025-10149-8}, \href{https://arxiv.org/abs/2304.00327}{\texttt{arXiv:2304.00327}}.

\bibitem[IM20]{ilmavirta2020unique}
J.~Ilmavirta and K.~M{\"o}nkk{\"o}nen.
\newblock Unique continuation of the normal operator of the x-ray transform and applications in geophysics.
\newblock {\em Inverse Problems}, 36(4), 2020.
\newblock Paper No. 045014, 23 pages. \href{https://mathscinet.ams.org/mathscinet/article?mr=4103726}{MR4103726}, \href{https://zbmath.org/1477.44002}{Zbl:1477.44002}, \href{https://doi.org/10.1088/1361-6420/ab6e75}{doi:10.1088/1361-6420/ab6e75}, \href{https://arxiv.org/abs/1909.05585}{\texttt{arXiv:1909.05585}}.

\bibitem[KRZ23]{MRZ2023pbiharmonic}
M.~Kar, J.~Railo, and P.~Zimmermann.
\newblock The fractional {$p$}-biharmonic systems: optimal {P}oincar{\'e} constants, unique continuation and inverse problems.
\newblock {\em Calc. Var. Partial Differential Equations}, 62(4), 2023.
\newblock Paper No. 130, 36 pages. \href{https://mathscinet.ams.org/mathscinet/article?mr=4568180}{MR4568180}, \href{https://zbmath.org/1516.35518}{Zbl:1516.35518}, \href{https://doi.org/10.1007/s00526-023-02468-9}{doi:10.1007/s00526-023-02468-9}, \href{https://arxiv.org/abs/2208.09528}{\texttt{arXiv:2208.09528}}.

\bibitem[Kek22]{Kekkonen2022IP}
H.~Kekkonen.
\newblock Consistency of {B}ayesian inference with {G}aussian process priors for a parabolic inverse problem.
\newblock {\em Inverse Problems}, 38(3), 2022.
\newblock Paper No. 035002, 29 pages. \href{https://mathscinet.ams.org/mathscinet/article?mr=4385425}{MR4385425}, \href{https://zbmath.org/1487.80018}{Zbl:1487.80018}, \href{https://doi.org/10.1088/1361-6420/ac4839}{doi:10.1088/1361-6420/ac4839}, \href{https://arxiv.org/abs/2103.13213}{\texttt{arXiv:2103.13213}}.

\bibitem[KRS21]{KRS21InstabilityMechanism}
H.~Koch, A.~R\"{u}land, and M.~Salo.
\newblock On instability mechanisms for inverse problems.
\newblock {\em Ars Inven. Anal.}, 2021.
\newblock Paper No. 7, 93 pages, \href{https://mathscinet.ams.org/mathscinet/article?mr=4462475}{MR4462475}, \href{https://zbmath.org/1482.35002}{Zbl:1482.35002}, \href{https://doi.org/10.15781/c93s-pk62}{doi:10.15781/c93s-pk62}, \href{https://arxiv.org/abs/2012.01855}{\texttt{arXiv:2012.01855}}.

\bibitem[KLW22]{KLW22CalderonFractionalWave}
P.-Z. Kow, Y.-H. Lin, and J.-N. Wang.
\newblock The {C}alder{\'{o}}n problem for the fractional wave equation: {U}niqueness and optimal stability.
\newblock {\em SIAM J. Math. Anal.}, 54(3):3379--3419, 2022.
\newblock \href{https://mathscinet.ams.org/mathscinet-getitem?mr=4434352}{MR4434352}, \href{https://zbmath.org/1492.35427}{Zbl:1492.35427}, \href{https://doi.org/10.1137/21M1444941}{doi:10.1137/21M1444941}, \href{https://arxiv.org/abs/2105.11324}{\texttt{arXiv:2105.11324}}.

\bibitem[KMS23]{KMS23GlobalUniquenessSemilinear}
P.-Z. Kow, S.~Ma, and S.~Sahoo.
\newblock An inverse problem for semilinear equations involving the fractional laplacian.
\newblock {\em Inverse Problems}, 39(9), 2023.
\newblock Paper No. 095006, 27 pages. \href{https://mathscinet.ams.org/mathscinet/article?mr=4629230}{MR4629230}, \href{https://zbmath.org/1525.35250}{Zbl:1525.35250}, \href{https://doi.org/10.1088/1361-6420/ace9f4}{doi:10.1088/1361-6420/ace9f4}, \href{https://arxiv.org/abs/2201.05407}{\texttt{arXiv:2201.05407}}.

\bibitem[KW23]{KW23FractionalNonlinear}
P.-Z. Kow and J.-N. Wang.
\newblock Inverse problems for some fractional equations with general nonlinearity.
\newblock {\em Res. Math. Sci.}, 10(4), 2023.
\newblock Paper No. 45, 33 pages. \href{https://mathscinet.ams.org/mathscinet/article?mr=4656890}{MR4656890}, \href{https://zbmath.org/1527.35498}{Zbl:1527.35498}, \href{https://doi.org/10.1007/s40687-023-00409-8}{doi:10.1007/s40687-023-00409-8}.

\bibitem[KW25]{KW25BayesSubdiffusion}
P.-Z. Kow and J.-N. Wang.
\newblock Consistency of {B}ayesian inference for a subdiffusion equation.
\newblock {\em SIAM/ASA J. Uncertain. Quantif.}, 13(3):1116--1144, 2025.
\newblock \href{https://mathscinet.ams.org/mathscinet/article?mr=4947191}{MR4947191}, \href{https://zbmath.org/8100478}{Zbl:8100478}, \href{https://doi.org/10.1137/24M1707419}{doi:10.1137/24M1707419}.

\bibitem[KS22]{KS2022collectivespontaneous}
J.~Kraisler and J.~C. Schotland.
\newblock Collective spontaneous emission and kinetic equations for one-photon light in random media.
\newblock {\em J. Math. Phys.}, 63(3), 2022.
\newblock Paper No. 031901, 22 pages. \href{https://mathscinet.ams.org/mathscinet/article?mr=4388726}{MR4388726}, \href{https://zbmath.org/1507.81214}{Zbl:1507.81214}, \href{https://doi.org/10.1063/5.0055171}{doi:10.1063/5.0055171}, \href{https://arxiv.org/abs/2104.12683}{\texttt{arXiv:2104.12683}}.

\bibitem[KS23]{KS2023kineticequations}
J.~Kraisler and J.~C. Schotland.
\newblock Kinetic equations for two-photon light in random media.
\newblock {\em J. Math. Phys.}, 64(11), 2023.
\newblock Paper No. 111903, 30 pages. \href{https://mathscinet.ams.org/mathscinet/article?mr=4662258}{MR4662258}, \href{https://zbmath.org/1531.78005}{Zbl:1531.78005}, \href{https://doi.org/10.1063/5.0106535}{doi:10.1063/5.0106535}, \href{https://arxiv.org/abs/2206.14336}{\texttt{arXiv:2206.14336}}.

\bibitem[Kwa17]{Kwa17FractionalEquivalent}
M.~Kwa\'snicki.
\newblock Ten equivalent definitions of the fractional {L}aplace operator.
\newblock {\em Fractional Calculus and Applied Analysis}, 20(1):7--51, 2017.
\newblock \href{https://mathscinet.ams.org/mathscinet-getitem?mr=3613319}{MR3613319}, \href{https://zbmath.org/1375.47038}{Zbl:1375.47038}, \href{https://doi.org/10.1515/fca-2017-0002}{doi:10.1515/fca-2017-0002}, \href{https://arxiv.org/abs/1507.07356}{\texttt{arXiv:1507.07356}}.

\bibitem[LL19]{LL19GlobalUniqueness}
R.-Y. Lai and Y.-H. Lin.
\newblock Global uniqueness for the fractional semilinear {S}chr{\"o}dinger equation.
\newblock {\em Proc. Amer. Math. Soc.}, 147(3):1189--1199, 2019.
\newblock \href{https://mathscinet.ams.org/mathscinet-getitem?mr=3896066}{MR3896066}, \href{https://zbmath.org/1406.35468}{Zbl:1406.35468}, \href{https://doi.org/10.1090/proc/14319}{doi:10.1090/proc/14319}, \href{https://arxiv.org/abs/1710.07404}{\texttt{arXiv:1710.07404}}.

\bibitem[Li24]{Li24NumericalFractionalCalderon2}
X.~Li.
\newblock Simultaneously reconstructing potentials and internal sources for fractional {S}chr{\"o}dinger equations.
\newblock {\em arXiv preprint}, 2024.
\newblock \href{https://arxiv.org/abs/2409.16716}{\texttt{arXiv:2409.16716}}.

\bibitem[Li25]{Li24NumericalFractionalCalderon1}
X.~Li.
\newblock A numerical method for reconstructing the potential in fractional {C}alder{\'o}n problem with a single measurement.
\newblock {\em Comput. Math. Appl.}, 183:256--270, 2025.
\newblock \href{https://mathscinet.ams.org/mathscinet/article?mr=4869090}{MR4869090}, \href{https://zbmath.org/8010446}{Zbl:8010446}, \href{https://doi.org/10.1016/j.camwa.2025.02.018}{doi:10.1016/j.camwa.2025.02.018}, \href{https://arxiv.org/abs/2409.16711}{\texttt{arXiv:2409.16711}}.

\bibitem[LL25]{LL25FractionalCalderonBook}
Y.-H. Lin and H.~Liu.
\newblock {\em Inverse Problems for {I}ntegro-differential {O}perators}, volume 222 of {\em Appl. Math. Sci.}
\newblock Springer, Cham, 2025.
\newblock \href{https://mathscinet.ams.org/mathscinet/article?mr=4923098}{MR4923098}, \href{https://zbmath.org/8014457}{Zbl:8014457}, \href{https://doi.org/10.1007/978-3-031-89142-7}{doi:10.1007/978-3-031-89142-7}.

\bibitem[LM72]{LM72NonhomogeneousBVPVOL1}
J.-L. Lions and E.~Magenes.
\newblock {\em Non-homogeneous boundary value problems and applications. {V}ol. {I}.}, volume 181 of {\em Die Grundlehren der mathematischen Wissenschaften}.
\newblock Springer-Verlag, New York-Heidelberg, 1972.
\newblock \href{https://mathscinet.ams.org/mathscinet/article?mr=0350177}{MR0350177}, \href{https://zbmath.org/0223.35039}{Zbl:0223.35039}.

\bibitem[Man01]{Man01instability}
N.~Mandache.
\newblock Exponential instability in an inverse problem for the {S}chr\"{o}dinger equation.
\newblock {\em Inverse Problems}, 17(5):1435--1444, 2001.
\newblock \href{https://mathscinet.ams.org/mathscinet-getitem?mr=1862200}{MR1862200}, \href{https://zbmath.org/0985.35110}{Zbl:0985.35110}, \href{https://doi.org/10.1088/0266-5611/17/5/313}{doi:10.1088/0266-5611/17/5/313}.

\bibitem[MV17]{valdinoci_biology}
A.~Massaccesi and E.~Valdinoci.
\newblock Is a nonlocal diffusion strategy convenient for biological populations in competition?
\newblock {\em J. Math. Biol.}, 74(1-2):113--147, 2017.
\newblock \href{https://mathscinet.ams.org/mathscinet/article?mr=3590678}{MR3590678}, \href{https://zbmath.org/1362.35312}{Zbl:1362.35312}, \href{https://doi.org/10.1007/s00285-016-1019-z}{doi:10.1007/s00285-016-1019-z}, \href{https://arxiv.org/abs/1503.01629}{\texttt{arXiv:1503.01629}}.

\bibitem[McL00]{McL00EllipticSystems}
W.~McLean.
\newblock {\em Strongly elliptic systems and boundary integral equations}.
\newblock Cambridge University Press, 2000.
\newblock \href{https://mathscinet.ams.org/mathscinet-getitem?mr=1742312}{MR1742312}, \href{https://zbmath.org/0948.35001}{Zbl:0948.35001}.

\bibitem[MRR{\etalchar{+}}53]{metropolis1953equation}
N.~Metropolis, A.~W. Rosenbluth, M.~N. Rosenbluth, A.~H. Teller, and E.~Teller.
\newblock Equation of state calculations by fast computing machines.
\newblock {\em J. Chem. Phys.}, 21(6):1087--1092, 1953.
\newblock \href{https://zbmath.org/1431.65006}{Zbl:1431.65006}, \href{https://doi.org/10.1063/1.1699114}{doi:10.1063/1.1699114}.

\bibitem[MNP21]{MNP21ConsistencyInversion}
F.~Monard, R.~Nickl, and G.~Paternain.
\newblock Consistent inversion of noisy non-{A}belian {X}-ray transforms.
\newblock {\em Comm. Pure Appl. Math.}, 74(5):1045--1099, 2021.
\newblock \href{https://mathscinet.ams.org/mathscinet/article?mr=4230066}{MR4230066}, \href{https://zbmath.org/7363259}{Zbl:7363259}, \href{https://doi.org/10.1002/cpa.21942}{doi:10.1002/cpa.21942}, \href{https://arxiv.org/abs/1905.00860}{\texttt{arXiv:1905.00860}}.

\bibitem[Nic20]{nickl2020bernstein}
R.~Nickl.
\newblock Bernstein-von {M}ises theorems for statistical inverse problems {I}: {S}chr{\"o}dinger equation.
\newblock {\em J. Eur. Math. Soc. (JEMS)}, 22(8):2697--2750, 2020.
\newblock \href{https://mathscinet.ams.org/mathscinet/article?mr=4118619}{MR4118619}, \href{https://zbmath.org/1445.62099}{Zbl:1445.62099}, \href{https://doi.org/10.4171/JEMS/975}{doi:10.4171/JEMS/975}, \href{https://arxiv.org/abs/1707.01764}{\texttt{arXiv:1707.01764}}.

\bibitem[Nic23]{Nickl23StatisticInverseProblem}
R.~Nickl.
\newblock {\em Bayesian non-linear statistical inverse problems}.
\newblock Zur. Lect. Adv. Math. EMS Press, Berlin, 2023.
\newblock \href{https://mathscinet.ams.org/mathscinet/article?mr=4604099}{MR4604099}, \href{https://zbmath.org/7713834}{Zbl:7713834}, \href{https://doi.org/10.4171/ZLAM/30}{doi:10.4171/ZLAM/30}.

\bibitem[Nic26]{Nickl2024BvMtime}
R.~Nickl.
\newblock Bernstein-von {M}ises theorems for time evolution equations.
\newblock {\em arXiv preprint}, 2026.
\newblock \href{https://arxiv.org/abs/2407.14781}{\texttt{arXiv:2407.14781}}.

\bibitem[NvdGW20]{NvdGW20JUQ}
R.~Nickl, S.~van~de Geer, and S.~Wang.
\newblock Convergence rates for penalized least squares estimators in {P}{D}{E} constrained regression problems.
\newblock {\em SIAM/ASA J. Uncertain. Quantif.}, 8(1):374--413, 2020.
\newblock \href{https://mathscinet.ams.org/mathscinet/article?mr=4074017}{MR4074017}, \href{https://zbmath.org/1436.62163}{Zbl:1436.62163}, \href{https://doi.org/10.1137/18M1236137}{doi:10.1137/18M1236137}, \href{https://arxiv.org/abs/1809.08818}{\texttt{arXiv:1809.08818}}.

\bibitem[OHKSW24]{OKSW2024nonlocalpartial}
E.~Orvehed~Hiltunen, J.~Kraisler, J.~C. Schotland, and M.~I. Weinstein.
\newblock Nonlocal partial differential equations and quantum optics: bound states and resonances.
\newblock {\em SIAM J. Math. Anal.}, 56(3):3802--3831, 2024.
\newblock \href{https://mathscinet.ams.org/mathscinet/article?mr=4753510}{MR4753510}, \href{https://zbmath.org/1542.35471}{Zbl:1542.35471}, \href{https://doi.org/10.1137/23M158142X}{doi:10.1137/23M158142X}.

\bibitem[PSU23]{PSU_book}
G.~P. Paternain, M.~Salo, and G.~Uhlmann.
\newblock {\em Geometric inverse problems---with emphasis on two dimensions}, volume 204 of {\em Cambridge Stud. Adv. Math.}
\newblock Cambridge University Press, Cambridge, 2023.
\newblock \href{https://mathscinet.ams.org/mathscinet/article?mr=4520155}{MR4520155}, \href{https://zbmath.org/1519.35005}{Zbl:1519.35005}, \href{https://doi.org/10.1017/9781009039901}{doi:10.1017/9781009039901}.

\bibitem[RO16]{Ros-Oton_survey}
X.~Ros-Oton.
\newblock Nonlocal elliptic equations in bounded domains: a survey.
\newblock {\em Publ. Mat.}, 60(1):3--26, 2016.
\newblock \href{https://mathscinet.ams.org/mathscinet/article?mr=3447732}{MR3447732}, \href{https://zbmath.org/1337.47112}{Zbl:1337.47112}, \href{https://doi.org/10.5565/PUBLMAT_60116_01}{doi:10.5565/PUBLMAT\_60116\_01}, \href{https://arxiv.org/abs/1504.04099}{\texttt{arXiv:1504.04099}}.

\bibitem[R{\"u}l21]{Rueland2021SingleMeasurement}
A.~R{\"u}land.
\newblock On single measurement stability for the fractional {C}alder{\'o}n problem.
\newblock {\em SIAM J. Math. Anal.}, 53(5):5094--5113, 2021.
\newblock \href{https://mathscinet.ams.org/mathscinet/article?mr=4311473}{MR4311473}, \href{https://zbmath.org/1476.35336}{Zbl:1476.35336}, \href{https://doi.org/10.1137/20M1381964}{doi:10.1137/20M1381964}, \href{https://arxiv.org/abs/2007.13624}{\texttt{arXiv:2007.13624}}.

\bibitem[RS18]{RS18Instability}
A.~R\"{u}land and M.~Salo.
\newblock Exponential instability in the fractional {C}alder\'{o}n problem.
\newblock {\em Inverse Problems}, 34(4):045003, 2018.
\newblock \href{https://mathscinet.ams.org/mathscinet-getitem?mr=3774704}{MR3774704}, \href{https://zbmath.org/6866428}{Zbl:06866428}, \href{https://doi.org/10.1088/1361-6420/aaac5a}{doi:10.1088/1361-6420/aaac5a}, \href{https://arxiv.org/abs/1711.04799}{\texttt{arXiv:1711.04799}}.

\bibitem[RS20]{RS20Calderon}
A.~R\"{u}land and M.~Salo.
\newblock The fractional {C}alder\'{o}n problem: {L}ow regularity and stability.
\newblock {\em Nonlinear Anal.}, 193:111529, 2020.
\newblock \href{https://mathscinet.ams.org/mathscinet-getitem?mr=4062981}{MR4062981}, \href{https://zbmath.org/1448.35581}{Zbl:1448.35581}, \href{https://doi.org/10.1016/j.na.2019.05.010}{doi:10.1016/j.na.2019.05.010}, \href{https://arxiv.org/abs/1708.06294}{\texttt{arXiv:1708.06294}}.

\bibitem[Tie98]{tierney1998note}
L.~Tierney.
\newblock A note on {M}etropolis-{H}astings kernels for general state spaces.
\newblock {\em Ann. Appl. Probab.}, 8(1):1--9, 1998.
\newblock \href{https://mathscinet.ams.org/mathscinet/article?mr=1620401}{MR1620401}, \href{https://zbmath.org/0935.60053}{Zbl:0935.60053}, \href{https://doi.org/10.1214/aoap/1027961031}{doi:10.1214/aoap/1027961031}.

\bibitem[Tri83]{Tri83FunctionSpace}
H.~Triebel.
\newblock {\em Theory of function spaces}, volume~78 of {\em Monogr. Math.}
\newblock Birkh{\"a}user Verlag, Basel, 1983.
\newblock \href{https://mathscinet.ams.org/mathscinet/article?mr=781540}{MR0781540}, \href{https://zbmath.org/0546.46027}{Zbl:0546.46027}, \href{https://doi.org/10.1007/978-3-0346-0416-1}{doi:10.1007/978-3-0346-0416-1}.

\bibitem[VAB{\etalchar{+}}96]{viswanathan1996}
G.~M. Viswanathan, V.~Afanasyev, S.~V. Buldyrev, E.~J. Murphy, P.~A. Prince, and H.~E. Stanley.
\newblock L{\'e}vy flight search patterns of wandering albatrosses.
\newblock {\em Nature}, 381(6581):413--415, 1996.
\newblock \href{https://doi.org/10.1038/381413a0}{doi:10.1038/381413a0}.

\end{thebibliography}
\end{document}